\documentclass[11pt]{article} 
\usepackage[T1]{fontenc} 
\usepackage[latin1]{inputenc}
\usepackage{lmodern} 
\usepackage[leqno]{amsmath}
\usepackage{amssymb,amscd,amsthm,amsmath,mathrsfs,yfonts,manfnt,pb-diagram,pb-xy}
\usepackage{bbm}

\usepackage{geometry} 
\usepackage{graphicx} 

\usepackage{xargs}
\usepackage[all]{xy}
\usepackage{booktabs} 
\usepackage{array} 
\usepackage{paralist} 
\usepackage{verbatim} 
\usepackage{subfig} 

\usepackage[nottoc,notlof,notlot]{tocbibind} 
\usepackage[titles,subfigure]{tocloft} 
\usepackage{hyperref}

\newcommand{\nablat}{\underline{\nabla }}
\newcommand{\D}{\underline{D}}

\newcommand{\A}{A}
\newcommand{\Aa}{\mathscr{A}}

\newcommand{\Bb}{\mathscr{B}}
\newcommand{\Bmod}{\mathcal{X}}
\newcommand{\bo}{\mathbb{B}}
\newcommand{\C}{\mathbbm{C}}
\newcommand{\Cc}{\mathscr{C}}
\newcommand{\clgen}{\gamma}
\newcommand{\E}{\mathbb{E}}
\newcommand{\modHilbL}{\mathcal{E}}
\newcommand{\HH}{H_3} 
\newcommand{\HHLie}{\mathfrak{h}_3} 
\newcommand{\Hh}{{H}}

\newcommand{\inj}{\hookrightarrow}

\newcommand{\Ll}{\mathcal{L}}

\newcommand{\mc}{\mathcal}
\newcommand{\modHilb}{E}

\newcommand{\N}{\mathbbm{N}}

\newcommandx*\QHM[2][1={},2={}]{{D^{#1}_{#2}}}
\newcommandx*\QHMl[2][1={},2={}]{{\mathscr{D}^{#1}_{#2}}}
\newcommandx*\modQHMl[2][1={},2={}]{{\mathscr{M}^{#1}_{#2}}}
\newcommand{\R}{\mathbbm{R}}
\newcommand{\Tt}{\mathcal{T}}

\newcommand{\PiCrG}{{B \rtimes _E \Z}}

\newcommand{\Z}{\mathbbm{Z}}

\newcommand{\eg}{\textsl{e.g.}\hspace{1ex}}
\newcommand{\ie}{\textsl{i.e.}\hspace{1ex}}
\newcommand{\locCit}{\textsl{loc.\ cit.\ }}

\DeclareMathOperator{\Diag}{Diag}
\DeclareMathOperator{\dom}{Dom}

\DeclareMathOperator{\Ev}{ev}
\DeclareMathOperator{\id}{id}
\DeclareMathOperator{\GNS}{GNS}
\DeclareMathOperator{\Graph}{graph}
\DeclareMathOperator{\Tr}{Tr}

\DeclareMathOperator{\Span}{Span}
\DeclareMathOperator{\Ext}{Ext}
\theoremstyle{definition}		
\newtheorem{Def}{Definition}[section]

\newtheorem{Ex}[Def]{Example}

\theoremstyle{plain}
\newtheorem{Prop}[Def]{Proposition}
\newtheorem{Th}[Def]{Theorem}
\newtheorem{Lem}[Def]{Lemma}
\newtheorem{Cor}[Def]{Corollary}

\theoremstyle{remark}
\newtheorem{Rem}[Def]{Remark}

\usepackage{scrpage2}      	
\ofoot{}					
\ifoot{}
\rehead{\pagemark}			
\cehead{O. Gabriel and M. Grensing}	
\rohead{\pagemark}			
\cohead{Gabriel, Grensing -- Spectral Triples and Generalized Crossed Products}

\title{Spectral Triples and Generalized Crossed Products}
\author{Olivier \textsc{Gabriel} \& Martin \textsc{Grensing}}

\begin{document}
\maketitle
\thispagestyle{empty}
\abstract{We give a construction for lifting spectral triples to crossed products by Hilbert bimodules. The spectral triple one obtains is a concrete unbounded representative of the  Kasparov product of the spectral triple and the Pimsner-Toeplitz extension associated to the crossed product by the Hilbert module. To prove that the lifted spectral triple is the above-mentioned Kasparov product, we rely on operator-$*$-algebras and connexions.}\\[5mm]
\medbreak
	
\tableofcontents

\section{Introduction} 

The starting point of this article is double: 
on the one hand, a standard way to construct new $C^*$-algebras is given by the crossed product construction; on the other hand, a ``smooth'' structure on a $C^*$-algebra is -- according to Connes' philosophy -- given by a so-called spectral triple. When the $C^*$-algebra is commutative, this corresponds -- under appropriate hypotheses -- to the choice of a $spin^c$ structure and associated Dirac operator, as is shown by the reconstruction theorem \cite{MR3032810}.  It is therefore natural to ask in which way the two notions can be combined, in other words, to study the permanence properties of spectral triples. Examples for spectral triples on crossed products by ordinary automorphisms have been around for a long time, recently they have been studied more intensively for example in \cite{DynSystBellMarcolli} and \cite{TrSpPiCr-Paterson}. 

The present paper aims to give a conceptual way for constructing spectral triples on a certain class of crossed product-like algebras, which contains the crossed products by $\Z$ and the (commutative) $S^1$-principal bundles -- see Proposition \ref{Prop:CommGCP}. This class fits within the general realm of Pimsner algebras \cite{PiCrG}, but our approach is more simply phrased within the setting of Abadie, Eilers and Exel in \cite{AbadieEE}. In fact, a construction of Pimsner in \cite{PiCrG} (which is based on a construction from Jones' paper \cite{MR696688}) shows that one can always write a Pimsner algebra as a generalized crossed product -- at the cost of changing the algebras and modules involved.

To sketch the basic idea, let us however start with the simplest case -- that of crossed products. Suppose for the moment that $\alpha$ is an isomorphism of a unital $C^*$-algebra $B$.  Recall first of all that Pimsner and Voiculescu proved  in \cite{MR587369} the existence of the so-called Pimsner-Voiculescu six-term exact sequence, which we state in its $K$-homological version:
\[\xymatrix{ K^0(B)\ar[d]^{\partial}&K^0(B)\ar[l]&K^0(B\rtimes_\alpha\mathbb{Z})\ar[l]\\
K^1(B\rtimes_\alpha\mathbb{Z})\ar[r]&K^1(B)\ar[r]&K^1(B)\ar[u]^{\partial}
}
\]
Here the vertical boundary maps are generalizations of the Fredholm index and are, in fact, given by the Kasparov product with the $KK^1$-class corresponding to the so-called Toeplitz-extension  of $B$ by $B\rtimes_\alpha\mathbb{Z}$ obtained by replacing the unitary in the crossed product by an isometry and identifying $KK^1$ with $\Ext$. 

Recalling that in the unbounded (or Baaj-Julg picture) of $KK$-theory \cite{MR715325} the $KK$-classes are actually given by spectral triples, it is natural to view this boundary map as a lifting procedure along the Toeplitz-extension for spectral triples on the base algebra $B$. And in fact, representing the crossed product canonically on $\ell^2(\mathbb{Z})\otimes \Hh$, where $\Hh$ is some faithful nondegenerate representation of $B$, one finds the completely natural operator
$N\otimes\gamma+1\otimes D$ associated to any unbounded representative $D$ with grading $\gamma$ of a $K^0$-class (here $N$ denotes the unbounded operator in $\ell^2(\mathbb{Z})$ defined by the identity function on $\Z$). Under some natural hypotheses, this operator is almost immediately seen to represent the Kasparov product of the class of the Toeplitz extension with the class $[D]\in KK_*(A,\mathbb{C})$ determined by $D$.

Suppose now that $E$ is a $B$-$B$-$C^*$-correspondence, \ie $E$ is a Hilbert module over $B$ and we are given a homomorphism $\phi:B\to\mathbb{B}_B(E)$. Suppose further that $E$ is full over $B$ and $\phi$ is injective. Then one may construct (see \cite{PiCrG} for details) an extension $\gamma_E$
$$
0\rightarrow I\rightarrow \mc{T}_E\rightarrow\mc{O}_E\rightarrow 0.
$$
Here the ideal $I$ is closely related to the base algebra $B$, and in good cases Morita equivalent to it.  Furthermore $A:=\mathcal{O}_E$ corresponds to the crossed product (if $E=B$ and $B$ acts on the left \textsl{via} an automorphism $\alpha$, then $\mathcal{O}_E\approx B\rtimes_\alpha\mathbb{Z}$). 

Using the same idea as above, we get a long exact sequence in $K$-homology and a boundary map $\partial_E:K^i(I)\to K^{i+1}(B)$ (which is given by the Kasparov product with the class of the extension). Thus we may try to lift spectral triples from $I$ to $A$ by just choosing an explicit representative of the Kasparov product of some fixed spectral triple on $I$ and the class $\partial_E$.

Now as mentioned above, already in the setting of crossed products there is a regularity condition a spectral triple on the base algebra $B$ has to satisfy in order to be ``liftable'' in a simple manner. In the more general setting of Pimsner algebras, this condition is more difficult to track down. For example, it is easily seen that for a covering with more than one leaf, the commutator with the Dirac on the base is not sufficiently well behaved to allow the simple lifting $N\otimes\gamma+1\otimes D$ indicated above. Moreover, the Cuntz algebra $\mathcal{O}_n$ is a Pimsner algebra but as it is purely infinite, it is traceless; this is another example which illustrates the difficulties involved in the construction of such liftings.

The regularity condition we use in this paper in order to construct liftings is the existence of certain connexions for the operator  $D$ on the base algebra. This connexion allows us to lift the spectral triple on the base algebra to an operator on the Pimsner algebra. This is certainly not the only possible approach, but the availability of techniques from \cite{SpecFlowKL}, which originate from \cite{arXiv:0904.4383}, have lead us to this approach. Ultimately, these results are based on \cite{Kucpub}. However, we need to adopt the results from \cite{SpecFlowKL} to the case of Kasparov products of odd and even unbounded modules which causes some technicalities deferred to the appendix.

Recall that by Kasparov's stabilization theorem \cite{MR587371} any countably generated Hilbert module over $B$ is a direct summand of the standard module $\Hh_B$.
Our main result concerning the Pimsner algebras $\mc{O}_E$ can then be formulated as follows (see Theorem \ref{Thm:Main}; for the precise definition of a Rieffel-spectral triple we refer the reader to section \ref{Section:RieffelSpecTrip}):
\begin{Th} 
Suppose that $B$ is unital and equipped with a Rieffel spectral triple $(\Hh,D_h)$ with grading $\gamma_h$ and that $E$ is a finitely generated projective Hilbert bimodule; assume $E$ is equipped with a two-sided Hermitian  $D_h$-connexion $\nabla$, then the lifting along the Pimsner-Toeplitz extension of of $(H,D_h)$ can be represented by the spectral triple 
$$(X\otimes_B\Hh,D_v\otimes \gamma+1\otimes_{\nablat}D_h)$$
where $X$ is the Hilbertmodule-completion of $B\rtimes_E\mathbb{Z}$ for the scalar product associated to the conditional expectation $\E(a):=\int_{S^1}\lambda.a\; d\lambda$ and $\nablat$ is a connexion explicitly constructed from $\nabla$ (compare \ref{Prop:ExtConn}).
\end{Th}
Here $\lambda.a$ refers to the natural circle-action on $B\rtimes_E\mathbb{Z}$.

The standard example (see Example \ref{Ex:BiMod}) of such a situation is given by a Hermitian line bundle $V$ over a locally compact Hausdorff base space $X$ by taking $E=\Gamma(V)$ and $B=C(X)$. The action by multiplication of $B$ on the left on $E$ may then be twisted by an isomorphism $\sigma$ of $B$ by setting $b\xi:=\sigma^{-1}(b)\xi$ and ${}_B\langle\xi|\eta\rangle:=\sigma(\langle\xi|\eta\rangle_B)$. For example, the quantum Heisenberg manifolds introduced by Rieffel \cite{RieffelDefQuant} fall into this category. This standard example motivates the notation $D_h$ (horizontal class) for the spectral triple on the base and $D_v$ (vertical class) for the Kasparov module determined by the Toeplitz extension.

\noindent The article is organized as follows:
\begin{itemize}
\item Section \ref{section:connexionsandopmod} contains preliminary results regarding operator $*$-algebras and operator $*$-modules, as well as some technical results concerning stability under holomorphic functional calculus. We put some emphasis on closability of derivations and connexions as this issue is sometimes overlooked. In fact, only under this closability condition the canonical operator-$*$-algebra closure (\cite{SpecFlowKL}, Proposition 2.6) of  the domain of a $*$-derivation on a $C^*$-algebra is a subalgebra of the $C^*$-algebra. Along the same lines, we construct operator-$*$-module closures for every connexion associated to a derivation (see Corollary \ref{Cor:FrameE1}).
\item Section \ref{Section:RieffelSpecTrip} recalls some basic facts regarding spectral triples and associated derivations. It then introduces a class of spectral triples associated to Lie group actions. These spectral triples were studied in \cite{TrSpLieGpGG}, where it is shown that they can be associated to ergodic Lie group actions and include many natural examples. However, the quantum Heisenberg manifolds are not included in this framework as shown in \locCit, and are one of the motivating examples for this article.
\item In Section \ref{Sec:crossed prods}, we give a brief account of a particularly well-behaved class of Pimsner algebras, the so-called generalized crossed products from \cite{AbadieEE}. These algebras carry a natural $S^1$-action, which however admits a non-empty fixed point subalgebra $B$; we thus do not obtain a spectral triple from this Lie group action, but rather a spectral triple ``with coefficients'', \ie a $KK$-class. This class is the vertical class (introduced properly in Definition \ref{Def:verticalclass}) which will be used later on in order to lift spectral triples from the base $B$ to the generalized crossed product. Along the way we give a complete characterization of commutative generalized crossed products.
\item The Section \ref{Section:Derext} contains the main technical tools of this article. We introduce two-sided connexions and prove how they can be extended to generalized crossed products. We show that under the condition of existence of a generalized frame and under some natural conditions on the connexion itself the extended connexion is closable -- a property which is essential as noted above. Using essentially the same idea, we also obtain a derivation on the generalized crossed product and thus a sub-$*$-algebra. 
\item In Section \ref{Section:TheLifting} we finally tackle the main theorem already stated above. We thus define the vertical class and calculate its products with a given Rieffel spectral triple on the base algebra. In order to do so, we show that the connexion constructed from two-sided connexion in Section \ref{Section:Derext} provides us with  a correspondence from the vertical class to the class of the spectral triple. We may then apply the (modified) criterion of Kaad and Lesch, which is proved in the appendix, to conclude. 
\item The last Section \ref{Sec:Examples} analyses the construction for the case of quantum Heisenberg manifolds. We show that the spectral triple obtained in this case coincides with a spectral triple constructed from scratch in  \cite{GeomQHM}. The result could be formulated in a somewhat more general framework, but we refrain from doing so here in order to keep to a moderate size. 
\item We have relegated to the appendix the somewhat technical task of adapting the results of Kaad and Lesch to the case of odd-even Kasparov products which we need in order to apply our results to  quantum Heisenberg manifolds. The modifications are mostly purely algebraic, depending basically on some yoga of graded tensor products. 
\end{itemize}

\section{Connexions and operator modules}\label{section:connexionsandopmod}

In this section, we study the notion of operator $*$-module introduced in \cite{SpecFlowKL}. We give explicit examples of such modules, based on derivations of the base algebra. 

We refer to \cite{MR1325694, Black} for missing details about Hilbert modules. We remind the reader of the following:
\begin{Def}
\label{Def:Frame}
A finite family $(\xi_j)_{j=1}^m$ in a (right) Hilbert $A$-module $E$ is called a (right) \emph{frame} if and only if
$$ \sum_{j=1}^m \xi_j \langle \xi_j, \cdot \rangle = \id_E .$$
\end{Def}

\begin{Rem}
\label{Rem:Frame}
A right Hilbert module $E$ admits a frame if and only if it is finitely generated (f.g.) and projective (see Proposition 3.9 p.89 of \cite{EltNCG}).
\end{Rem}

We also need the notion of $C^*$-correspondence. Since several definitions of ``correspondences'' exist in the literature, we make explicit the definition we use. It is called a ``Hilbert bimodule'' in \cite{SpecFlowKL} -- but we keep this term for a more symmetric structure (see Definition \ref{Def:HilbBimod} below).
\begin{Def}
Given $C^*$-algebras $A$ and $B$, an $A$-$B$ \emph{$C^*$-correspondence} $E$ is
\begin{itemize}
\item
a right Hilbert module over $B$ whose scalar product we denote $\langle \cdot , \cdot \rangle _B$ or $\langle \cdot , \cdot \rangle $ when the context is clear;
\item
a map $\pi \colon A \to \bo_B(E)$, where $\bo_B(E)$ is the set of maps $T \colon E \to E$ such that there is a $T^* \colon E \to E$ with $\forall \xi, \eta \in E$, 
$$
 \langle \xi, T \eta \rangle = \langle T^* \xi, \eta \rangle .
$$ 
In particular, $T$ is right $B$-linear and bounded.
\end{itemize}
\end{Def}
We write $a \xi$ instead of $\pi (a) \xi$ when this notation is unambiguous.

\medbreak

There is a straightforward notion of tensor product of $C^*$-correspondences (see \cite{Black}, II.7.4 p.147 or \cite{MR1325694}):
\begin{Def}
\label{Def:TensorProduct}
Given a (right-) $A$-Hilbert module $E$ and an $A$-$B$ $C^*$-correspondence $F$, the inner tensor product over $A$, denoted $E \otimes _A F$, is the $B$-Hilbert module obtained from the quotient of the algebraic tensor product $E \odot_\C F$ by the subspace generated by
$$ \{ \xi a \otimes \eta - \xi \otimes \pi (a) \eta| \xi \in E, \eta \in F, a \in A \} $$
by completing it for (the norm induced by) the scalar product:
$$
\langle \xi \otimes \eta, \xi' \otimes \eta' \rangle _B := \langle \eta, \langle \xi, \xi' \rangle_B \cdot \eta' \rangle_B .
$$
If $E$ is actually a $C$-$A$ $C^*$-correspondence, then the resulting tensor product $E \otimes _A F$ is a $C$-$B$ $C^*$-correspondence.

\smallbreak

We can iterate this construction to obtain $A$ $C^*$-correspondences $$E \otimes _A E \otimes _A \cdots \otimes_A E$$ for any positive integer $k$ (the tensor product contains $k$ factors) which we denote by  $E^{\otimes k}$.
\end{Def}

\begin{Def}\label{Def:Bmod}
Let $A$ be a $C^*$-algebra. A \emph{Banach $A$-bimodule} $\Bmod$ is a Banach space equipped with continuous left- and right actions from $A$.

In this paper, we focus more precisely on two types of such Banach $A$-bimodules:
\begin{enumerate}[(i)]
\item
the case of a $A$-$A$ $C^*$-correspondence $\Bmod$;
\item
given two $A$-$B$ $C^*$-correspondences $F_1$ and $F_2$, set $\Bmod := \bo_B(F_1 \to F_2)$. This is a Banach $A$-bimodule for the left- and right-actions are provided post- and pre-composition with $\pi_j(a)\in \bo_B(F_j)$. It is readily checked that these actions are continuous.
\end{enumerate}
\end{Def}
Here (i) is actually just a special case of (ii), as is shown by the following construction:
\begin{Def}
\label{Def:LinkingAlg}
If $F_1$ and $F_2$ are two $A$-$B$ $C^*$-correspondences, $F_1 \oplus F_2$ their $A$-$B$ $C^*$-correspondence direct sum, \ie the inner sum $B$-Hilbert module equipped with the diagonal $A$-action. The algebra $\bo_B(F_1 \oplus F_2)$ decomposes naturally into a direct sum which we can write in matrix notation:
$$ 
\bo_B(F_1 \oplus F_2) = 
\begin{pmatrix}
\bo_B(F_1) & \bo_B(F_2,F_1) \\
\bo_B(F_1,F_2) & \bo_B(F_2)
\end{pmatrix}.
$$
In particular, if $A$ is unital and $\Bmod$ is an $A$-$A$ $C^*$-correspondence, then setting $F_1 := A$ and $F_2 := \Bmod$, we recover $\bo_A(A \to \Bmod) = \Bmod \subseteq \bo_A(A \oplus \Bmod)$.
\end{Def}
Indeed, $A$ is unital and any $\varphi \in \bo_A(A \to \Bmod)$ is fully determined by $\varphi(1)= \xi \in \Bmod$. This bijection preserves the norm and therefore gives an isomorphism. 

In the sequel, we will only write $\bo(F)$ instead of $\bo_A(F)$ when the context is clear.

\begin{Rem}
This is just a variation on the theme of \emph{linking algebras} as introduced by Rieffel in \cite{MR0367670} (see for instance \cite{Black}, II.7.6.9 p.152 for an overview). 
\end{Rem}

\begin{Def}
\label{Def:Deriv}
A \emph{derivation} $\delta$ on a $C^*$-algebra $A$ with values in a Banach $A$-bimodule $\Bmod$ is a linear map defined on a dense $*$-subalgebra $\Aa \subseteq A$ with values in $\Bmod$ which satisfies for all $a, b \in \Aa$,
$$ \delta(a b) = \delta(a) b + a \delta(b) .$$
Such a derivation is \emph{closable} if the closure of its graph in $A \times \Bmod$ is the graph of a function. In other words, any sequence $a_n \in \Aa$ such that
$$
a_n \to 0\;\;\; \text{ and }\;\;\;
\delta (a_n) \to  y
$$
(convergence in $A$ and $\Bmod$, resp.) satisfies $y = 0$.

\medbreak

If $\Bmod$ admits an involution which is compatible with the action of $A$, $\delta$ is a \emph{$*$-derivation} if $\delta(a)^* = \delta(a^*)$ whenever the involved terms are well-defined.
\end{Def}
Given a derivation $\delta$ in the sense of Definition \ref{Def:Deriv}, defined on a subalgebra $\Aa$, we want to construct the associated ``$C^1$-functions'' (more precisely almost everywhere Lipschitz functions, see Lemma 1 in \cite{NCG}). More formally, we construct an \emph{operator $*$-algebra} $A_1$ (see Definition 2.3 p.8 of \cite{SpecFlowKL}), notion which in turn depends on that of \emph{operator space}. This later concept is well-understood (see \eg \cite{SubspacesRuan, CPOABlecher}), suffice it to say here that these are Banach spaces which admit a suitable extension of their norms to finite matrices and that they are precisely (norm-)closed subspaces of $C^*$-algebras. Naturally associated to these structures are \emph{completely bounded morphisms}, \ie morphisms that extend to matrices of arbitrary size while keeping a bounded norm. An \emph{operator $*$-algebra}, as introduced first in Definition 3.2.3 of \cite{arXiv:0904.4383} and Definition 3.3 of \cite{CPInvolOAIvankov}:
\begin{Def}
\label{Def:OpEtAlg}
An \emph{operator $*$-algebra} $A_1$ is an operator space s.t.
\begin{itemize}
\item
its multiplication $m \colon A_1 \times A_1 \to A_1$ is completely bounded;
\item
there is a completely bounded involution $\dagger \colon A_1 \to A_1$ -- the extension to matrices being provided by transposing matrices (and applying $\dagger$ entrywise);
\end{itemize}
\end{Def}
An important example of such operator $*$-algebra is provided by the following adaptation of Proposition 2.6 in \locCit
\begin{Prop}
\label{Prop:DefB1}
Let $F_1$ and $F_2$ be two $B$-$C$ $C^*$-correspondences with faithful left-actions $\pi_1$, $\pi_2$ of $B$ and a derivation $\delta$ from $B$ to $\bo_C(F_1 \to F_2)$ which is defined on the dense subalgebra $\Bb \subseteq B$. Define an algebra morphism from $\Bb$ to $\bo(F_1 \oplus F_2)$ by
$$
\rho(b) = 
\begin{pmatrix}
\pi_1(b) & 0 \\ \delta(b) & \pi_2(b)
\end{pmatrix}.
$$
The completion $B_1$ of $\Bb$ for $\| b\|_1 := \| \rho(b) \|$ is a dense subalgebra of $B$ if and only if $\delta$ is closable. In this case, $B_1$ has the following properties:
\begin{enumerate}
\item
Both the inclusion $B_1 \inj B$ and $b \mapsto \delta(b)$ are completely bounded.
\item
$B_1$ is stable under holomorphic calculus.
\item
If $\Bb$ is a $*$-algebra, $F_1 = F_2 =: F$ and $\delta(b^*) = U \delta(b)^* U$ for some unitary $U \in \bo_C(F)$ which \emph{commutes} with $\pi(b)$ for $b \in \Bb$, then $B_1$ is an operator $*$-algebra.
\end{enumerate}
\end{Prop}

\begin{proof}
We start by the proof of the inclusion $B_1 \inj B$. Considering $\rho$ as above, by definition $B_1$ is included in $\bo(F_1 \oplus F_2)$. We have two continuous linear maps from $B_1$ to $\bo(F_1)$ and $\bo(F_1 \to F_2)$ given on elements of $\Bb$ by $\phi(b) = \pi_1(b)$ and $\psi(b) = \partial(b)$. By definition of $B_1$, these maps are continuous for $\| \cdot \|_1$ and therefore extend to $B_1$. 

\smallbreak

However, if $\delta$ is not closable, $\phi$ is not injective. Indeed, in this case there is a sequence $b_n \in  \Bb$ s. t. $b_n \to 0$ and $\partial(b_n) \to y \neq 0$. Denote $b$ the limit of this Cauchy sequence in $B_1$, then $\phi( b ) = 0$ but $\psi( b ) = y \neq 0$.

\smallbreak

Conversely, if $\delta$ is closable, since $\pi_1$ is faithful, the image of $\Bb$ by the map $b \mapsto (\phi(b), \psi(b))$ is (isomorphic to) the graph of $\delta$. $B_1$ is then (isomorphic to) the closure of this graph in $B \times \bo(F_1 \to F_2)$. If this closure is still the graph of a function, then $\phi$ is injective and $B_1 \inj B$. Moreover, $\Bb$ is dense and $\Bb \subseteq B_1$ -- which proves the density of $B_1$ in $B$. 

The stability under holomorphic functional calculus is then a straightforward consequence of the ``standard result'' Lemma 2 p.247 of \cite{NCG}, 6.$\alpha$.

\medbreak

The inclusion $B_1 \inj B$ is completely bounded: the tensor product of $B_1$ by finite matrices comes with the restriction of the unique $C^*$-norm on $\bo(F_1 \oplus F_2) \otimes M_N(\C)$. The map from $B_1 \otimes M_N(\C)$ to $B \otimes M_N(\C)$ is thus contractive.

\medbreak

For the last point, we refer to Proposition 2.6 in \cite{SpecFlowKL}.
\end{proof}

We now introduce connexions associated to derivations, \ie our definition is slightly different from  \cite{NCG} III.1. Definition 5 p.227.
\begin{Def}
\label{Def:Conn}
Given a (right) $A$-Hilbert module $E$ and a derivation $\delta\colon\Aa\to \Bmod$ on $A$ with values in an $A$-$A$ $C^*$-correspondence $\Bmod$, an associated  (right) \emph{connexion} on $E$ is a linear map $\nabla\colon\modHilbL\to E\otimes_B\Bmod$ defined on a dense subset $\modHilbL \subseteq E$ such that
\begin{itemize}
\item
$\modHilbL$ is a right $\Aa$-module and
\item
for all $a\in \Aa$, $\xi \in \modHilbL$, 
\begin{equation}
\label{Eqn:Connexion}
 \nabla (\xi a) = \nabla (\xi) a + \xi \otimes \delta (a).
\end{equation}
\end{itemize}
The connexion $\nabla $ is called \emph{closable} if $\delta$ is closable and for any sequence $\xi_n \in \modHilbL$ which satisfies $\xi_n \to 0$ and $\nabla (\xi_n) \to \eta$, we have $\eta = 0$.

If $\nabla $ is closable, we denote the domain of its closure by $\overline{\dom(\nabla)}$.
\end{Def}
%
%
%

\begin{Rem}
It is readily checked that $\overline{\dom(\nabla )}$ is a $B_1$-module. We use the symbol $\nabla$ to denote both the connexion and its closure.
\end{Rem}

\begin{Prop}
\label{Prop:Connexion}
Any connexion $\nabla $ on a Hilbert $B$-module $E$, associated to a derivation $\delta$, induces a derivation $\partial $ on $\bo_B(E)$ with values in $\bo_B(E \to E \otimes \Bmod)$ defined by
\begin{equation}
\label{Eqn:DefPartial}
 \big( \partial (T)  \big)(\xi) = \nabla (T(\xi)) - (T \otimes \id_X)(\nabla(\xi)) .
\end{equation}
Moreover, if $\nabla $ is closed, defined on a dense subset $E_1 := \overline{\dom(\nabla)} \subseteq E$, $E$ is finitely generated projective and $\langle E_1, E_1 \rangle \subseteq B_1$ then $\partial $ is densely defined and is closable.

\smallbreak

The domain of its closure is:
\begin{equation}
\label{Eqn:DomPartial}
 \overline{\dom(\partial )} := \{ T \in \bo_B(E) | T(\overline{\dom(\nabla )}) \subseteq \overline{\dom(\nabla )} \},
\end{equation}
which is therefore stable under holomorphic functional calculus.
\end{Prop}
\begin{proof}
If $\dom(\nabla )$ is the domain of the connexion $\nabla $, then we set:
$$ \dom(\partial ) := \{ T \in \bo_B(E) | \, T (\dom(\nabla )) \subseteq \dom(\nabla ) \} .$$
and define $\partial T$ by \eqref{Eqn:DefPartial}.

It remains to prove that $\partial $ is a derivation:
\begin{align*}
&\partial (T T')(\xi) 
= \nabla ( T T'(\xi)) - (TT' \otimes 1)(\nabla \xi) \\
=& \nabla ( T T'(\xi)) - (T \otimes 1)\nabla(T'(\xi)) + (T \otimes 1) \nabla (T'(\xi)) - (TT' \otimes 1)\nabla \xi \\
=& (\partial T)(T'(\xi)) + (T \otimes 1) \big[ \partial (T')(\xi) \big].
\end{align*}
Whenever $\partial T$ is defined, it is $\Bb$-linear:
\begin{align*}
&(\partial T)(\xi b) = \nabla (T(\xi b)) - (T \otimes 1)(\nabla(\xi b))\\
= &\nabla (T(\xi)) b + T(\xi) \delta(b) - (T \otimes 1)( \nabla (\xi)) b - T(\xi) \delta(b) \\
=& \nabla (T(\xi)) b - (T \otimes 1)( \nabla (\xi)) b = \partial (T)(\xi) b.
\end{align*}

Under the additional assumptions, we prove the density of $\dom(\partial )$ : since $E$ is f.g. projective, any element $T \in \bo_B(E)$ can be written
$$ T = \sum_{j=1}^m \eta_j \langle \xi_j, \cdot \rangle ,$$
where $\eta_j$ and $\xi_j$ are elements of $E$. Since $\modHilbL$ is dense in $E$, we can find ``perturbations'' $\eta_j'$ and $\xi_j'$ in $\modHilbL$ of $\eta_j$ and $\xi_j$ respectively. It is then easy to prove using the Hilbert module norm that if $\eta_j'$ and $\xi_j'$ are chosen ``close'' to $\eta_j$ and $\xi_j$, then
$$ T' := \sum_{j=1}^m \eta_j' \langle \xi_j', \cdot \rangle $$
is ``close'' to $T$. Furthermore, $\langle \modHilbL, \modHilbL \rangle \subseteq B_1$ ensures that $T' \in \dom(\partial )$. This concludes the proof that $\partial $ is densely defined.

\medbreak

If we apply the previous argument to $T = \id_E$, we get an element $T' \in \dom(\partial )$ which is invertible in $\bo_B(E)$. This implies in particular that $(\eta_j')$ generates $E$ over $B$. This property ensures that we can extend $\partial T$ -- which is \textsl{a priori} only defined on $\dom(\nabla )$ -- by $B$-linearity to an element of $\bo(E \to E \otimes \Bmod)$. Indeed,  if $V$ is the inverse (in $\bo(E)$) of $T$, then 
$$ \xi = T' V(\xi) = \sum_{j} \eta_j' \langle \xi'_j, V(\xi) \rangle .$$
This provides a \emph{unique} decomposition of $\xi$ in terms of $\eta'_j$. We then extend $\partial T$ by setting:
$$ (\partial T)(\xi) := \sum_{j} (\partial T)(\eta_j') \langle \xi'_j, V(\xi) \rangle .$$
Since for any fixed $j$, the map $\xi \mapsto \langle \xi'_j, V(\xi) \rangle$ is continuous, the map $\partial T$ thus extended is continuous and coincides with the ``natural definition'' of $\partial T$ on the dense subset $\dom(\nabla )$. This continuous extension is therefore unique and $\partial T \in \bo(E \to E \otimes \Bmod)$.

\medbreak

We now prove that if $\nabla $ is closable, then $\partial $, defined naturally on 
\begin{equation}
\label{Eqn:DomPartial2}
\dom(\partial ) := \left\{ T \in \bo_C(E) \middle| T(\overline{\dom(\nabla )}) \subseteq \overline{\dom(\nabla )} \right\},
\end{equation}
 is closable: if we have a sequence $T_n$ in $\dom(\partial )$ such that both $T_n \to 0$ and $\nabla T_n \to V$ (convergence in $\bo_B(E)$ and $\bo_B(E \to E \otimes \Bmod)$, respectively), then $V = 0$. 

\smallbreak

The proof is very similar to that of Lemma \ref{Lem:CloseDeriv}: we consider $\xi \in \overline{\dom(\nabla )}$ and denote $\eta := \nabla (\xi) \in E \otimes \Bmod$. Since $T_n \to 0$, we have $T_n \nabla (\xi) \to 0$ and consequently
$$ (\partial T_n)(\xi) = \nabla(T_n (\xi)) - (T_n \otimes 1) \nabla(\xi) \to V(\xi) .$$
The sequence $T_n(\xi) \in \overline{\dom(\nabla )}$ tends to $0$ and $\nabla (T_n(\xi))$ admits the limit $V(\xi)$ for $n \to 0$. As $\nabla $ is closed, we must have $V(\xi) = 0$. This is true on the dense subspace $\overline{\dom(\nabla )} $ of $E$, therefore $V = 0$ and $\partial $ it closable.

\medbreak

Finally, the inclusion of sets given by \eqref{Eqn:DomPartial2} is actually an equality:
consider
$$
\Graph(\partial  ) := \{ (T, \partial  T), T ( \overline{\dom(\nabla )}) \subseteq \overline{\dom(\nabla )}  \} \subseteq \bo_{B}(E) \times \bo_B(E \to E \otimes \Bmod) .
$$
A Cauchy sequence $(T_n, \partial T_n)$ in $\Graph(\partial )$ is a sequence such that both $T_n$ and $\partial T_n$ converge in $\bo_B(E)$ and $\bo_B(E \to E \otimes \Bmod)$, respectively. Denote $T_\infty$ the limit of $T_n$ in $\bo_B(E)$. For any $\xi \in E_1$, we have $T_k \xi \to T_\infty \xi$. Moreover, the definition \eqref{Eqn:DefPartial} can be read:
$$ \partial (T_k(\xi)) = (\partial T_k)(\xi) + (T_k \otimes 1)(\partial \xi) .$$
By hypothesis, $\partial T_k$ is norm convergent and $(T_k \otimes 1)(\partial \xi) \to (T_\infty  \otimes 1) (\partial \xi)$, hence $\partial (T_k(\xi))$ converges in $E \otimes \Bmod$. This implies that $T_k(\xi)$ actually converges in $E_1$ and thus $T_\infty(\xi) \in E_1$. This in turn implies that $T_\infty (E_1) \subseteq E_1$, \ie $T_\infty \in \overline{\dom(\partial )}$. 

\bigbreak

The last property of the Lemma is an immediate consequence of our previous results and of the ``standard result'' Lemma 2 p.247 of \cite{NCG}, 6.$\alpha$.
\end{proof}

For our purposes, the notion of \emph{operator $*$-module} is the essential companion of operator $*$-algebra as introduced in Definition \ref{Def:OpEtAlg} above. We call \emph{standard module} $A_1^\infty $ over an operator $*$-algebra $A_1$ the closure of the finite column-vectors with entries in $A_1$. We are now ready to state (compare \cite[Definition 3.4]{SpecFlowKL}):
\begin{Def}
\label{Def:OpEtMod}
An \emph{operator $*$-module} $E_1$ over an operator $*$-algebra $A_1$ is an operator space together with 
\begin{itemize}
\item
a completely bounded (right) action of $A_1$ on $E_1$;
\item
a completely bounded pairing $\langle \cdot , \cdot \rangle$ with values in $A_1$, which satisfies the same algebraic conditions as for a Hilbert module;
\item
a completely bounded selfadjoint idempotent $P \colon A_1^\infty \to A_1^\infty $ s.t. $P A_1^\infty $ is isomorphic to $E_1$ (as operator space).
\end{itemize}
\end{Def}

\begin{Rem}
The abstract notion underlying our constructions is more clearly expressed in terms of the following alternative definitions: call \emph{derivation} on an algebra $A$ a pair of linear maps $(\pi,\delta)$ where $\pi:A\to  B$ is a homomorphism to an algebra $B$ and $\delta:\Aa\to B$ is a linear map defined on a subalgebra $\Aa \subseteq A$ such that for all $a, a' \in \Aa$:
$$ \delta(a a') = \delta(a) \pi(a') + \pi(a) \delta(a').$$
One obtains an associated notion of connexion as follows: a $\delta$-connexion on $E$ is given by an $\Aa$-submodule $\modHilbL\subseteq E$ and linear maps $\pi_E:E\to D$, $\nabla\colon\modHilbL\to D$ such that for all $a\in \Aa$ and $\xi\in \modHilbL$:
$$\nabla(\xi a)=\nabla(\xi)\pi_A(a)+\pi_E(\xi)\delta(a).$$
 $\modHilbL$ will be referred to  as the domain of $\nabla$. This yields a submodule and subalgebra of $M_2(D)$ by identifying
\begin{align*}&a\approx\begin{pmatrix} \pi_A(a)&0\\\delta(a)&\pi_A(a)\end{pmatrix} &\xi\approx\begin{pmatrix} \pi_E(\xi)&0\\\nabla(\xi)&\pi_E(\xi)\end{pmatrix}\end{align*}
which thus carry natural structures of an operator algebra and operator module. However, such definition would require a careful study of the dependence on $D$, $\delta$ and $\nabla$ which we want to avoid for the moment.
\end{Rem}
The previous Proposition \ref{Prop:Connexion} has the following consequences:
\begin{Cor}
\label{Cor:FrameE1}
If 
\begin{itemize}
\item
$\delta$ and $\nabla $ are respectively a densely defined and closable derivation on $B$ and a $\delta$-connexion on a finitely generated projective $C^*$-correspondence $E$;
\item
we have for $E_1 := \overline{\dom(\nabla )}$, the domain of the closure of $\nabla$, that $\langle E_1, E_1 \rangle \subseteq B_1$,
\end{itemize}
then 
\begin{enumerate}[(i)]
\item
there is a frame of $E$ consisting of elements of $E_1$,
\item
both the inclusion $E_1 \inj E$ and the map $E_1\to E,\;\xi \mapsto \nabla (\xi)$ are completely bounded.
\end{enumerate}
\end{Cor}

\begin{Rem}
This implies in particular that $E_1$ is an operator module in the sense of \cite{SpecFlowKL}, Definition 3.1. Moreover, the frame of elements of $E_1$ induces a projection with entries in $B_1$ -- as required by Definition 3.4 in \cite{SpecFlowKL}.
\end{Rem}

\begin{proof}
Regarding the point (i), we use Remark \ref{Rem:Frame} to a obtain a (finite) frame $\xi_i$ in $E$, \ie
$$ \sum \xi_i \cdot \langle \xi_i, \cdot \rangle = \id_E. $$
Just as in the proof of Proposition \ref{Prop:Connexion}, we consider a perturbation $\zeta_j$ of $\xi_i$ which is ``close enough'' to $\xi_i$. We thus obtain an operator $\sum_{j=1}^m \zeta_j \cdot \langle \zeta_j, \cdot \rangle =: T \in \bo_{B}(E)$ which is ``close enough'' to $1_E$ -- and in particular invertible in $\bo_B(E)$. It is also clear that $T$ sends $\overline{\dom(\nabla )}$ to itself.

\smallbreak

We can then consider $V := T^{-1/2}$, which is defined by holomorphic functional calculus and is therefore in $\overline{\dom(\partial )}$ -- by Proposition \ref{Prop:Connexion}. By construction, $V$ belongs to the algebra generated by $T$ and since $T = T^*$, $V$ commutes with $T$ and $V = V^*$. We can evaluate:
$$ V T V = \id_E = \sum_{j} (V \zeta_j) \cdot \langle V \zeta_j, \cdot  \rangle $$
where we used that $V^* = V$. This equation precisely means that $\big( V \zeta_j \big)$ is a frame for $E$. The equality \eqref{Eqn:DefPartial} proves that $V \zeta_j \in E_1$ and therefore concludes the proof of point (i).

\bigbreak

Point (ii) requires to introduce $F := B \oplus E$. It is easy to see that $\nabla' := \delta \oplus  \nabla $ is a connexion on $F$. Since $\delta$ and $\nabla $ are closable, so is $\nabla'$. In the same way, since $E$ is f.g. projective, so is $F$. The compatibility of $F_1$ and $B_1$ with the scalar product follows from that of $E_1$ and $B_1$. We then rely on the linking algebra as in Definition \ref{Def:LinkingAlg}: applying Proposition \ref{Prop:Connexion} and Proposition \ref{Prop:DefB1} to $\nabla'$, we get a completely bounded inclusion $\bo(F)_1 \inj \bo(F)$ and a complete bounded map extending the derivation $\partial '$ on $\bo(F)$.

We can restrict and corestrict the injection to $E \subseteq \bo(F)$ as in Definition \ref{Def:LinkingAlg}, thereby obtaining a completely bounded inclusion $E_1 \inj E$. Consider now the restriction of $\partial '$ to $E \subseteq \bo(F)$: if $\xi \in E \simeq \bo(B \to E)$, we can evaluate $\partial '(\xi)$ on $1 \in B_1$:
$$ \partial'(\xi)(1) = \nabla ( \xi 1) - \xi \delta(1) = \nabla (\xi) . $$
By restricting and corestricting $\partial '$, we hence prove that $\xi \mapsto \nabla (\xi)$ from $E_1$ to $E \otimes \Bmod$ is completely bounded.
\end{proof}

\section{Dirac-Rieffel operators}\label{Section:RieffelSpecTrip}
\label{Sec:DiracLieGp}
\subsection{Spectral triples}
In this article, our main object of interest will be \emph{spectral triples} in the following sense:
\begin{Def} 
\label{Def:TrSp} Let $A$ be a unital $C^*$-algebra. An odd \emph{spectral triple}, also called odd \emph{unbounded Fredholm module}, is a triple $(\pi, \Hh, D)$ where:
\begin{itemize}
\item
$\Hh$ is a Hilbert space and $\pi \colon A\to \bo(\Hh)$ a faithful $*$-representation of $A$ by bounded operators on $\Hh$;
\item
a selfadjoint unbounded operator $D$ -- which we will call the \emph{Dirac operator} -- defined on the domain $\dom(D)$;
\end{itemize}
such that
\begin{enumerate}[(i)]
\item
$(1+D^2)^{-1/2}$ is compact,
\item
the subalgebra $\Aa$ of all $a\in A$ such that $[\pi(a),D]$ is bounded is dense in $A$.
\end{enumerate}
An \emph{even spectral triple} is given by the same data, but we further require that a grading $\gamma$ be given on $\Hh$ such that (i) $A$ acts by even operators, (ii) $D$ is odd.
\end{Def}
We recall for the readers' convenience that a commutator of bounded operator $T$ with an unbounded operator $D$ is said to be bounded (written $[D,T]\in\mathbb{B}(\Hh)$ if $T\dom(D)\subseteq \dom(D)$ and $[D,a]$ is a bounded operator on its domain. As is common practice, we denote the closure of $[D,a]$ again by $[D,a]$. We caution the reader that this condition is often obscured (see \eg \cite{BlackK}), and that this definition is indeed the only reasonable one for the definition of $KK$-theory (compare \cite{MR2679393}, Section 4).

	Spectral triples provide the standard examples of closed derivations on $\Aa$ with values in $\bo(\Hh) = \bo_\C(\Hh)$:
\begin{Lem}
\label{Lem:CloseDeriv}
Let $(\pi,\Hh,D)$ be a spectral triple on $A$, then $D$ determines a closable derivation on $\Aa$ with values in $\bo(\Hh)$ by $\delta_D(a) := [D, a]$. Its closure is the derivation $\bar\delta_D$ with domain $\dom(\bar\delta_D)=\{a\in A\big|\; [D,a]\in\bo(\Hh)\}$ defined by $\bar\delta_D(a)=[D,a]$.
\end{Lem}

\begin{Rem}
It follows that $\delta'(a) := i [D, a]$ is a $*$-derivation.
\end{Rem}
\begin{proof}
The proof is similar to that of Proposition \ref{Prop:Connexion}. The operator $D$ is  selfadjoint and thus closed. Let $a\in\dom(\bar\delta_D)$, \ie there exists a sequence $a_n \in \Aa$ with $a_n \to a$ in $A$ and $\delta_D(a_n)$ convergent in $\bo(\Hh)$. By hypothesis, $a_n \dom(D) \subseteq \dom(D)$. Given any $\xi \in \dom(D)$, $[D,a_n]\xi$ converges  and $a_nD\xi\to aD\xi$, thus $D(a_n \xi)$ is convergent. As $D$ is closed and $a_n\xi\to a\xi$, $a\xi\in \dom(D)$ and $Da\xi=\lim Da_n\xi$. It follows that for all $\xi\in\dom(D)$:
$$\bar\delta_D(a)\xi=\lim \bar\delta_D(a_n)\xi=\lim Da_n\xi-a_nD\xi=[D,a]\xi$$
thus $[D,a]\in\bo(\Hh)$ and $\bar\delta_D(a)=[D,a]$.
\end{proof}

\smallbreak
It is natural to define $ \Omega^1_D$, the \emph{$1$-forms associated to $D$} (see \cite{NCG}, VI.1. Proposition 4 p.549), as the closed linear span in $\bo(\Hh)$:
\begin{equation}
\label{Eqn:DefOmega1}
\Omega^1_D := \Span \Big\{ a_j^0 [D, a_j^1] \Big| a^0_j, a^1_j \in \Aa \Big\}.
\end{equation}

The space of $1$-forms $\Omega^1_D$ is clearly an $\Aa$-bimodule \textsl{via} the representation $\pi \colon \Aa \to \bo(\Hh)$. It is a Banach bimodule of the type mentioned in the point (ii) of Definition \ref{Def:Bmod}. In this context, we obtain a ``concrete'' (or ``represented'') version of Definition \ref{Def:Conn}:
\begin{Def}\label{Def:OperatorConnexion} (compare Chapter 6.1, Definition 8 of \cite{NCG})
Let $(\pi,\Hh,D)$ be a spectral triple on a unital $C^*$-algebra $A$, $E$ a Hilbert $A$-module. Denote by $\delta_D$ the derivation on $\Aa$ with values in $\Omega^1_D$ defined by $\delta_D(a)=[D,a]$. We call \emph{$D$-connexion on $E$} a connexion on $E$ with respect to $\delta_D$.
\end{Def}
\subsection{Rieffel Spectral Triples}

We recall briefly the definition of the \emph{complexified Clifford algebra} of dimension $n$ (see \cite{MR1031992} for  further details):
\begin{Def}
\label{Def:Clifford}
For $n \in \N$, we denote $\C l(n)$ the universal unital $C^*$-algebra generated by $n$ selfadjoint elements $e_j$ which satisfy the relations:
\begin{equation}
\label{Eqn:Spin}
 e_j e_k + e_k e_j = 2 \delta_{jk} .
\end{equation}
A $\Z/2 \Z$-grading on $\C l(n)$ is induced by the automorphism $h$ defined by $h(e_i) = - e_i$. 

In the rest of this article, we will consider a \emph{Clifford module} $S$, \ie a module $S$ equipped with a representation $\pi_{\C l(n)}$ of $\C l(n)$ and denote by $\clgen_i$ the operators $i \pi_{\C l(n)}(e_j) \in \bo(S)$. To unclutter notation, we will most of the time suppress the representation $\pi_{\C l(n)}$. For simplicity, we take $S$ to be finite dimensional.
\end{Def}

\begin{Rem}
The operators $\clgen_j$ satisfy:
\begin{align}
\label{Eqn:RelF}
\clgen^* &= - \clgen 
&
\clgen_j \clgen_k + \clgen_k \clgen_j &= - 2 \delta_{jk}.
\end{align}
\end{Rem}

The present work focuses on a particular form of Dirac operators:
\begin{Def}
\label{Def:RieffelTrSp}
We say that $(\pi, \Hh, D)$ is a \emph{Rieffel spectral triple} if 
\begin{itemize}
\item
$\Hh = \Hh_0 \otimes S$ for some Hilbert space $\Hh_0$ and Clifford module $S$;
\item
the operator $D$ is a \emph{Dirac-Rieffel operator}, \ie with respect to the previous decomposition, it can be written:
\begin{equation}
\label{Eqn:Dirac}
 D = \sum_{j=1}^n \partial _j \otimes \clgen_j ,
\end{equation}
with $\clgen_j$ as in Definition \ref{Def:Clifford} and unbounded operators $\partial _j$ on $\Hh$ such that $\partial _j^* = - \partial _j$;
\item
$ \partial _j^\Aa(a):=[\partial _j, a]$ for all $a \in \Aa$ defines $*$-derivations on $A$ with values in $A$.
\end{itemize}
\end{Def}

\begin{Rem}
A direct consequence of the above definition is: for all $a \in \Aa$,
\begin{equation}
\label{Eqn:CommRieffelTrSp}
 [D, a] = \sum_{j} \partial _j^\Aa(a) \otimes \clgen_j .
\end{equation}
\end{Rem}

In our previous article \cite{TrSpLieGpGG}, we proved that such Dirac-Rieffel operators can be constructed from ergodic actions of compact Lie groups -- see Theorem 5.4 therein. In the present article, we want to investigate the ``permanence properties'' of this class of spectral triples.

\section{\texorpdfstring{Operator $*$-modules}{Operator *-modules}}

We study connexions associated to Rieffel spectral triples as introduced in Definition \ref{Def:RieffelTrSp}. For the rest of this section, let $(\pi, \Hh, D)$ be a Rieffel spectral triple and $\delta_D:=[D,\;\cdot\;]$ the associated derivation (Lemma \ref{Lem:CloseDeriv}). As we will see, in this setting, there is a particularly simple notion of \emph{Hermitian connexion}. At the root of this lie the following results:
\begin{Lem}
\label{Lem:DecompNabla}
 The space of $1$-forms $\Omega^1_D$ satisfies:
\begin{equation}
\label{Eqn:CharOmega1D}
 \Omega^1_D \subseteq \pi(B) \otimes_\C \langle \clgen_1, \ldots , \clgen_n \rangle \subseteq \bo(\Hh)
\end{equation}
and has a natural $B$-Hilbert module structure.

\smallbreak

If $\nabla $ is a connexion defined on $\modHilbL$ and associated to $\delta_D$, then $\nabla(\xi) = \sum_{j} \nabla _j(\xi) \otimes \clgen_j$ for  certain ``components'' $\nabla _j$ with the following properties:
\begin{itemize}
\item
$\nabla _j \colon \modHilbL \to \modHilb$;
\item for all $j$, the following equality holds for $\xi \in \modHilbL$ and $b \in \Bb$:
\begin{equation}
\label{Eqn:DerivNabla}
\nabla _j(\xi b) = \nabla _j(\xi) b + \xi \, \partial _j(b).
\end{equation}
\end{itemize}
Conversely, given maps $\nabla _j$ defined on $\modHilbL$ which satisfy Equation \eqref{Eqn:DerivNabla}
\begin{equation}
\label{Eqn:RecompConnexion}
\nabla (\xi) := \sum_{} \nabla _j(\xi) \otimes \clgen_j .
\end{equation}
defines a concrete connexion associated to $D$.

Finally, $\nabla $ is closable if and only if all its components are closable.
\end{Lem}


\begin{proof}
The proof of this lemma is essentially algebraic. The inclusion \eqref{Eqn:CharOmega1D} is a direct consequence of Equation \eqref{Eqn:CommRieffelTrSp}. We can then define a $B$-valued scalar product on $\Omega^1_D$ by setting:
$$
\langle b \otimes \clgen_j, b' \otimes \clgen_k \rangle  = \delta_{j,k} b^* b' = (\id_B \otimes \Tr)\big( (b \otimes \clgen_j)^* (b' \otimes \clgen_k) \big)
$$ 
where $\Tr$ is the unique normalised trace on the finite dimensional matrix algebra $B(S)$. Hence in this particular case, the Banach bimodule $\Bmod$ actually fits within case (i) of Definition \ref{Def:Bmod}.

An immediate consequence of \eqref{Eqn:CharOmega1D} is the inclusion:
$$ 
E \otimes _B \Omega^1_D = \Span \big\{ \xi \otimes [D, b] \big| \xi \in E, b \in B_1 \big\} \subseteq E \otimes_\C B(S).
$$
For any given $\xi$, there are $n$ unique elements $\xi_j \in \modHilb$ such that
$$ \nabla (\xi) = \sum_{} \xi_j \otimes \clgen_j . $$
Moreover, since the family $\clgen_j$ is free in $B(S)$, these $\xi_i$ are uniquely defined. It then suffices to define the maps $\nabla _j$ by $\xi \mapsto \xi_j$ and Equation \eqref{Eqn:Connexion} reads:
$$ \sum \nabla _j(\xi b) \otimes \clgen_j = \sum \nabla _j(\xi) b \otimes \clgen_j + \xi \sum \partial _j(b) \otimes \clgen_j .$$
Since $\clgen_j$ is free in $B(S)$, it is equivalent to say that for all $j$, 
$$ \nabla _j(\xi b) = \nabla _j(\xi) b + \xi \otimes \partial _j(b) $$
where the equalities take place in $\modHilbL$.

\medbreak

Now, assume that the family $\nabla _j$ satisfies Equation \eqref{Eqn:DerivNabla}, let us show that \eqref{Eqn:RecompConnexion} defines a connexion associated to $D$:
$$
\nabla (\xi b) = \sum \nabla _j(\xi b) \otimes \clgen_j = \sum (\nabla _j(\xi) b + \xi \partial _j(b)) \otimes \clgen_j = \nabla (\xi) b + \xi [D, b].
$$
Finally, the result on closability is a straightforward consequence of the decomposition of $\nabla $ into components.
\end{proof}

\begin{Rem}
\label{Rem:ProdScalNabla}
Using Lemma \ref{Lem:DecompNabla}, we can give a meaning to the scalar products
\begin{align*}
\langle \xi, \nabla \eta \rangle &:= \sum \langle \xi,  \nabla _j(\eta) \rangle \otimes \clgen_j
&
\langle \nabla \xi, \eta \rangle &:= -\sum \langle \nabla _j(\xi), \eta \rangle \otimes \clgen_j.
\end{align*}
It is readily checked that $\langle \nabla \xi, \eta \rangle ^* = \langle \eta, \nabla \xi \rangle $ -- and this accounts for the minus sign in $\langle \nabla \xi, \eta \rangle $.\\
Comparing the above definition with the pairing appearing in \cite{SpecFlowKL} (4.2), we see that both expressions are compatible. 
\end{Rem}
We proceed with \emph{Hermitian connexions}, following \cite{SpecFlowKL}, Definition 4.3 
and \cite{NCG} VI.1 p.553: 
\begin{Def}
\label{Def:HermConnexion}
Consider a connexion $\nabla$ on $\modHilb$, associated to $\delta_D:=[D,\;\cdot\;]$ as in Lemma \ref{Lem:DecompNabla}, which is defined on $\modHilbL$ such that the $B$-valued scalar product on $\modHilb $ restricts to $\modHilbL$ in a compatible way, \ie $\langle \modHilbL, \modHilbL \rangle \subseteq B_1$. Then $\nabla $ is \emph{Hermitian} if for all $\xi, \eta \in \modHilbL$
\begin{equation}
\label{Eqn:CompScalProd2}
[ D, \langle \xi, \eta \rangle ] = \langle \xi, \nabla (\eta) \rangle - \langle \nabla (\xi), \eta \rangle.
\end{equation}
\end{Def}

\begin{Rem}
\label{Rem:StabE1}
It is easily seen, by decomposing with respect to the free family $(\clgen_j)$ of $B(S)$ that \eqref{Eqn:CompScalProd2} is equivalent to for all $\xi, \eta \in \modHilbL$ and $j$,
\begin{equation}
\label{Eqn:CompScalProdj}
\langle \xi, \nabla _j(\eta) \rangle  + \langle \nabla _j(\xi), \eta \rangle = \partial _j \big \langle \xi, \eta \rangle.
\end{equation}
Moreover, if $\nabla $ is Hermitian on $\modHilbL$, then $\langle \modHilb_1, \modHilb_1 \rangle \subseteq B_1$ and \eqref{Eqn:CompScalProd2} is satisfied on $\modHilb_1$. 
\end{Rem}

\begin{Prop}
\label{Prop:OpStarMod}
If 
\begin{itemize}
\item
$\nabla $ is a closable connexion on a f.g. projective $C^*$-correspondence $E$, 
\item
$\dom(\nabla )$ is dense in $E$ and for $E_1 := \overline{\dom(\nabla )}$, $\langle E_1, E_1 \rangle \subseteq B_1$,
\item
$\nabla $ is associated to the derivation $\delta_D$ induced by a Rieffel spectral triple,
\item
$\nabla $ is Hermitian,
\end{itemize}
then $E_1$ is an operator $*$-module (for the standard involution $\dag =*$ -- see Definition 3.2 p.12 of \cite{SpecFlowKL}).
\end{Prop}

\begin{proof}
All hypotheses are satisfied to apply Cor. \ref{Cor:FrameE1} to $\nabla $ and $E$. This existence of a frame suffices to obtain the isomorphism with $P A^\infty $ as required in \cite{SpecFlowKL} p.13. Indeed, consider the finite matrix $P := \Big( \langle \zeta_j , \zeta_k \rangle \Big)_{jk}$. This is a selfadjoint operator:
\begin{multline*}
 \langle P a_i, b_i \rangle = \sum_{i=1}^m \sum_{j=1}^m (P_{ij} a_j)^* b_i = \sum_{i,j =1}^m \langle \zeta_i, \zeta_j a_j\rangle^* b_i \\
= \sum_{i,j =1}^m \langle  \zeta_j a_j, \zeta_i\rangle b_i = \sum_{i,j=1}^m a_j^* \langle \zeta_j, \zeta_i \rangle b_i = \langle a_j, P b_j \rangle .
\end{multline*}
$P$ is also an idempotent:
$$
\sum_{k=1}^m \langle \zeta_j , \zeta_k \rangle  \langle \zeta_k , \zeta_l \rangle 
=
\sum_{k=1}^m \Big\langle \zeta_j , \zeta_k   \langle \zeta_k , \zeta_l \rangle \Big\rangle = \langle \zeta_j , \zeta_l \rangle 
$$
since $(\zeta_j)$ is a frame. $E_1$ is isomorphic to $P A_1^m$, in the sense that there are maps $\Phi \colon E_1 \to P A_1^m$ and $\Psi \colon P A_1^m \to E_1$ defined by:
\begin{align*}
\Phi(\xi) &= \big( \langle \zeta_j, \xi \rangle \big)_j
&
\Psi\big( a_j \big) &= \sum \zeta_j a_j
\end{align*}
such that $\Phi \circ \Psi = \id_{P A_1^m}$ and $\Psi \circ \Phi = \id_{E_1}$. 

Finally, the scalar product induced on $P A_1^m$ coincides with the one induced by the isomorphism with $E_1$:
$$ 
\langle \Psi(a_j), \Psi(b_j) \rangle = \left\langle \sum \zeta_j a_j, \sum \zeta_k b_k \right \rangle = \sum a_j^* \langle \zeta_j, \zeta_k \rangle  b_k = \sum a_j^* P_{j k} b_k = \sum a_j^* b_j
$$
since $\sum_{k} P_{jk} b_k = b_j$.

\bigbreak

To prove the complete boundedness property of the pairing, we appeal to the Hilbert space $\Hh$. We set $E' := B \oplus E$ and equip it with the closed connexion $\delta_D \oplus \nabla $. We can apply Proposition \ref{Prop:Connexion} to $\bo(E') =: B'$, thereby getting a derivation $\partial' $ on $B'$ with values in $\bo(E' \to E' \otimes \Omega^1_D)$. 

Introduce $\Hh' := E' \otimes _B \Hh = \Hh \oplus E \otimes \Hh$, whose decomposed tensors we write $x \oplus \eta \otimes y$. We can realise $\bo(E' \to E' \otimes \Omega^1_D)$ inside $B(\Hh')$. Indeed, given $T \in \bo(E' \to E' \otimes \Omega^1_D)$ we let it act on $\Hh'$ by:
$$ T( x \oplus \eta \otimes y) = T(\eta) y $$
where $T(\eta) \in E' \otimes \Omega^1_D$ acts on $y \in \Hh$ \textsl{via} the inclusion $\Omega^1_D \subseteq \bo(\Hh)$. This enables us to apply Proposition \ref{Prop:DefB1} with $\bo(\Hh' \oplus \Hh')$ as underlying algebra.

\medbreak

To summarise, $B$ and $E$ sit inside $B'$ and are represented on $B(\Hh')$ \textsl{via} the representation $\pi$ of the spectral triple, which we extend by $\pi(\xi)(x \oplus \eta \otimes y) = \xi \otimes x$ and
$$
\pi(b)(x \oplus \eta \otimes y) = \pi_0(b)x \oplus \pi_1(b) (\eta \otimes y) = b x \oplus (b \eta) \otimes y.
$$
The two representations extend naturally to $\Omega^1_D$ and $E \otimes _B \Omega^1_D$: it suffices to let $\id_{\Hh_0} \otimes \clgen_j$ act on $\Hh = \Hh_0 \otimes S$.  This gives a meaning to the action of $\nabla (\xi) = \sum_{} \nabla _j(\xi) \otimes \clgen_j$ on $\Hh'$. Since we are using faithful representations of $C^*$-algebras, the norms on $E_1$ and $B_1$ coincide with those obtained from Corollary \ref{Cor:FrameE1}.

%
%
%
%
%
%
%
%
%

The expressions $\pi(\xi)$, $\partial'(\pi(\xi)) = \nabla (\pi(\xi))$ (abbreviated to $\nabla (\xi)$) and $\nabla (\xi)^*$ all have a meaning in $B(\Hh')$. It is readily checked that $ \pi(\xi)^* \pi(\eta) = \pi_0(\langle \xi, \eta \rangle )$ and we can then write:
\begin{multline*}
\pi(\xi)^* \nabla \eta - (\nabla \xi)^* \pi(\eta) = \pi(\xi)^* \left( \sum \pi(\partial _j(\eta)) \otimes \clgen_j \right) - \sum_{} \pi(\partial _j(\xi) \otimes \clgen_j)^* \pi(\eta) \\
= \sum_{} \pi_0 \big( \langle \xi, \partial _j \eta \rangle \otimes \clgen_j \big) + \sum_{} \pi_0 \big( \langle \partial _j(\xi), \eta \rangle \otimes \clgen_j \big) = \pi_0 \big( [ D, \langle \xi, \eta \rangle ] \big).
\end{multline*}
This leads to 
\begin{multline*}
\begin{pmatrix}
\pi_0 \big(\langle \xi, \eta \rangle \big) & 0 \\
\pi_0 \Big(\delta \big( \langle \xi, \eta \rangle \big) \Big) & \pi_0 \big( \langle \xi, \eta \rangle \big)
\end{pmatrix}
=
\begin{pmatrix}
\pi(\xi)^* \pi(\eta) & 0 \\
\pi(\xi)^* \nabla \eta - (\nabla \xi)^* \pi(\eta)  & \pi(\xi)^* \pi(\eta)
\end{pmatrix}
\\
=
\begin{pmatrix}
\pi(\xi)^* & 0 \\ -\nabla (\xi)^* & \pi(\xi)^*
\end{pmatrix}
\begin{pmatrix}
\pi(\eta) & 0 \\ \nabla (\eta) & \pi(\eta)
\end{pmatrix}
\end{multline*}
The leftmost matrix is (isometric to) the one we used to define $B_1$ in Proposition \ref{Prop:DefB1}. To make contact with the norm on $E_1$ we remark that:
$$ 
\begin{pmatrix}
\xi^* & 0 \\ -\nabla (\xi)^* & \xi^*
\end{pmatrix}
=
\begin{pmatrix}
\xi & -\nabla (\xi) \\ 0 & \xi
\end{pmatrix}
^*
= 
\left(
\begin{pmatrix}
0 & -1 \\ 1 & 0
\end{pmatrix}
\begin{pmatrix}
\xi & 0 \\ \nabla (\xi) & \xi
\end{pmatrix}
\begin{pmatrix}
0 & 1\\ -1 & 0
\end{pmatrix}
\right)^*
$$
which proves that $\| \langle \xi, \eta \rangle \|_{1} \leqslant \| \xi \|_{1} \| \eta \|_{1}$. 
\end{proof}

We conclude this section with the following proposition, which by exception doesn't require finiteness of the original $C^*$-correspondence $E$. It is a sort of converse to  Corollary \ref{Cor:FrameE1}. It requires the following
\begin{Def}
\label{Def:FrameInf}
Let $E$ be a $C^*$-correspondence which is projective. A (possibly infinite) sequence $(\xi_j)$ is a ``frame'' for $E$ if 
$$ T_N := \sum_{k =1}^N \xi_j \langle \xi_j, \cdot \rangle  $$
satisfies $T_N \xi \to \xi$ for any $\xi \in E$ (convergence in $E$).
\end{Def}
Examples of such frames will appear naturally in connexion with generalized crossed products by finitely generated projective modules below (see Proposition \ref{Prop:ExtConn}).

\begin{Prop}
\label{Prop:FrameClosed}
Let 
\begin{itemize}
\item
$\partial $ be a derivation on $B$ associated to a Rieffel spectral triple,
\item
and $\nabla $ be an associated Hermitian connexion defined on the dense and $B_1$-stable subset $\modHilbL \subseteq E$.
\end{itemize}
If $(\xi_j) \in \modHilbL$ is a ``frame'' for $E$ then the connexion $\nabla $ is closable.
\end{Prop}

\begin{Rem}
This is in particular true for the Grassmannian connexion (see \cite{SpecFlowKL}, Section 4.1). 
\end{Rem}

\begin{proof}
We start from a sequence $\zeta_n$ of elements in $\dom(\nabla )$ such that $\zeta_n \to 0$ and $\nabla (\zeta_n) \to \eta$. We want to prove that $\eta = 0$.

From Definition \ref{Def:FrameInf}, we see that it suffices that $\langle \xi_j, \eta \rangle = 0$ for all $j$ to obtain $\eta = 0$. We estimate $\langle \xi_j, \eta \rangle$:
$$
 \langle \xi_j, \nabla (\zeta_n) \rangle = \partial ( \langle \xi_j, \zeta_n \rangle ) + \langle \nabla (\xi_j), \zeta_n \rangle
$$
Since $\nabla (\xi_j)$ is a fixed quantity and $\zeta_n \to 0$, we have $\langle \nabla (\xi_j), \zeta_n \rangle \to 0$. Moreover, since both $\xi_j$ and $\zeta_n$ are elements of $\dom(\nabla )$, $b_n = \langle \xi_j, \zeta_n \rangle \in \dom(\partial )$ and this sequence tends to $0$ in norm. Our previous computation ensures that $\partial b_n \to \langle \xi_j, \eta \rangle $. Since $\partial $ is closed, we get $\eta = 0$.
\end{proof}

\section{Generalised crossed products: a review}\label{Sec:crossed prods}

Our aim in this paper is to construct a spectral triple for Generalised Crossed Products (GCP), we therefore introduce this notion.

\smallbreak

\begin{Def}
\label{Def:HilbBimod}
A \emph{Hilbert bimodule} $E$ over a $C^*$-algebra $A$ is a bimodule $E$ over $A$ such that
\begin{itemize}
\item
$E_A$ (right $A$-module) is a right-Hilbert module over $A$ for the scalar product $\langle \cdot , \cdot \rangle _A$;
\item
${}_A E$ is a left-Hilbert module over $A$ for $ {}_A \langle \cdot , \cdot \rangle$;
\item
for all $\xi, \eta, \zeta \in E$,
$$ {}_A \langle \xi, \eta \rangle \zeta = \xi \langle \eta, \zeta \rangle _A .$$
\end{itemize}
\end{Def}

\begin{Rem}
%
%
Several notions of ``Hilbert bimodule'' can be found in the litterature. Our convention is different from the one used in \cite{PiCrG}, for instance. 
\end{Rem}

%

To illustrate this notion, consider the following simple example: let $X$ be a  compact space, and denote by $B := C(X)$ the $C^*$-algebra of complex continuous functions on $X$. We further assume that we have a (locally trivial) line bundle $\Ll \to X$. According to Serre--Swan's Theorem (see for instance \cite{EltNCG} Theorem 2.10 p.59), the $B$-module $E$ of sections of $\Ll$ is a f.g. projective bundle over $B$. 

If moreover, $E$ is equipped with a fibrewise scalar product $( \cdot | \cdot ) \colon E_x \times E_x \to \C$ (\emph{Hermitian line bundle}), then we can define a \emph{Hilbert module} structure on $E$ (see Definition II.7.1.1 p.137 
of \cite{Black}). Adding an automorphism of $X$, we can define a Hilbert \emph{bimodule} structure on $E$:
\begin{Ex}
\label{Ex:BiMod}
Consider $B = C(X)$, a commutative unital $C^*$-algebra and $\sigma$, an automorphism of $B$. We denote by $\tilde{\sigma} \colon X \to X$, the homeomorphism of $X$ such that $\sigma(b)(x) = b(\tilde{\sigma}(x))$. A Hilbert module $E$ over $B$ yields a Hilbert bimodule structure by setting, for all $b \in B, \xi \in E$:
\begin{align*}
b \cdot \xi &:= \xi \, \sigma^{-1}(b) 
&
{}_B \langle \xi, \eta \rangle &:= \sigma \big( \langle \eta, \xi \rangle _B \big).
\end{align*}
\end{Ex}
It is clear that Hilbert bimodules are in particular $C^*$-correspondences. Moreover, it is easily checked that given two Hilbert bimodules, the tensor product $E \otimes _B F$ (see Definition \ref{Def:TensorProduct} above) is also a Hilbert bimodule. The left-scalar product is defined by:
$$
{}_B \langle \xi \otimes \eta, \xi' \otimes \eta' \rangle := {}_B \langle \xi \cdot {}_B \langle \eta, \eta' \rangle, \xi' \rangle .
$$
Iterating the construction, we obtain $B$-Hilbert bimodules $ E^n := E \otimes _B E \otimes _B \cdots \otimes_B E$ for $n \in \N$, where the tensor product contains $n$-factors.

\medbreak

In the particular case described in Example \ref{Ex:BiMod}, the left- and right-products are linked by the relation: for $\xi_1, \ldots , \xi_n \in E$ and $b \in B$,
\begin{multline}
\label{Eqn:CommN}
b \cdot (\xi_1 \otimes \xi_2 \otimes \cdots \otimes \xi_n) = (\xi_1 \cdot \sigma^{-1}(b)) \otimes \xi_2 \otimes \cdots \otimes \xi_n \\
= \xi_1 \otimes (\xi_2 \cdot \sigma^{-2}(b)) \otimes \cdots \otimes \xi_n =  \xi_1 \otimes \xi_2  \otimes \cdots \otimes (\xi_n\cdot \sigma^{-n}(b)).
\end{multline}
Let $A$ be a $C^*$-algebra equipped with a pointwise continuous $S^1$-action $\gamma$. We call such $\gamma$ a \emph{gauge action} and it yields \emph{spectral subspaces} $A^{(k)}$, $k \in \Z$, defined by:
\begin{equation}
\label{Eqn:SpecSubspace}
A^{(k)} := \{T \in A | \forall z \in S^1 \subseteq \C, \gamma_z(T) = z^k T \} .
\end{equation}
$T \in A^{(k)}$ is called a \emph{gauge-homogeneous} element of \emph{gauge} $k$. It is readily checked that $A^{(0)}$ is a $C^*$-algebra and that given any $k \in \Z$, $A^{(k)}$ has a natural structure of Hilbert bimodule over $A^{(0)}$.

\smallbreak

We now define GCP, which are our main topic of interest. The following appeared as Definition 2.1 in \cite{AbadieEE}:
\begin{Def}
\label{Def:RepHilbBimod}
Given a Hilbert bimodule $E$ over a $C^*$-algebra $B$, a \emph{representation of $E$} on a $C^*$-algebra $C$ is a $C^*$-algebra morphism $\pi$ together with a linear map $S \colon E \to C$ such that for all $\xi, \zeta \in E$,
\begin{align*}
(i)& \hspace{0.25cm} S(\xi)^* S(\zeta) = \pi \left(\langle \xi , \zeta \rangle_B \right) 
&
(ii)& \hspace{0.25cm} S(\xi) \pi(b) = S(\xi b) \\
(iii)& \hspace{0.25cm}\pi(b) S(\xi) = S(b \xi)
&
(iv)&\hspace{0.25cm} S(\xi) S(\zeta)^* = \pi \left({}_B \langle \xi , \zeta \rangle \right).
\end{align*}
\end{Def}

\begin{Def}
\label{Def:PiCrG}
Let ${}_B E_B$ be a $B$-Hilbert bimodule. The \emph{Generalized Crossed Product} (GCP)  $\PiCrG $ of $B$ by $E$ is the universal $C^*$-algebra generated by the representations of $E$, in the sense that (see \cite{Black}, II.8.3):
\begin{itemize}
\item
there is a representation $(\pi_E, S_E)$ of $E$ on $\PiCrG $;
\item 
$\PiCrG $ is generated by $\pi_E(B) \cup S_E(E)$;
\item 
for any representation $(\pi, S)$ of $E$ on $C$, there is a $C^*$-morphism $\rho_{(\pi, \Tt)} \colon \PiCrG  \to C$ such that the following diagram commutes:
$$
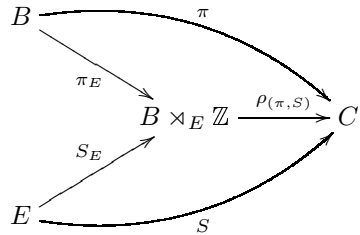
\xymatrix@C=1.2cm{
B \ar[dr]_{\pi_E} \ar@/^{6mm}/[drr]^{\pi} & & \\
& \PiCrG  \ar[r]^-{\rho_{(\pi,S)}}& C\\
E \ar[ru]^{S_E} \ar@/_{6mm}/[rru]_{S}& &
}
$$
\end{itemize}
\end{Def}
We say that $F \in \PiCrG $ is \emph{algebraic} if it is in the involutive algebra generated by $\pi_E(A)$ and $S_E(E)$ -- without taking the closure. By definition, algebraic elements are dense in the GCP.
\begin{Ex}
\label{Ex:GCP-X}
Given a compact space $X$, a Hermitian line bundle $\Ll$ over $X$ and an automorphism $\sigma$ of $B =C(X)$, we call \emph{GCP associated to $(X, \Ll, \sigma)$} the GCP generated by $B$ and the Hilbert bimodule $E$ constructed from the sections of $\Ll$, as described in Example \ref{Ex:BiMod}.
\end{Ex}
The notion of GCP is closely related to that of (Cuntz-)Pimsner algebra \cite{PiCrG}. An important characterization of GCP is the following Theorem 3.1 of \cite{AbadieEE}:
\begin{Th}
\label{Thm:DefGCP}
A $C^*$-algebra $A$ equipped with a gauge action $\gamma$ is the GCP of $A^{(0)}$ by $A^{(1)}$ if and only if it is generated as a $C^*$-algebra by $A^{(0)}$ and $A^{(1)}$.
\end{Th}
Using this characterisation, we can prove that commutative GCP correspond essentially to principal $S^1$-bundles:
\begin{Prop}
\label{Prop:CommGCP}
If $A$ is a unital commutative $C^*$-algebra carrying an $S^1$-action such that 
\begin{itemize}
\item
 $E:=A^{(1)}$ generates $A$ as $C^*$-algebra,
\item 
 $E$ is a f.g. projective module over $B := A^{(0)}$,
\end{itemize}
then 
\begin{enumerate}[(i)]
\item
$B = C(X)$ for some compact space $X$,
\item 
$E$ is the module of sections of some \emph{line} bundle $\Ll \to X$,
\item 
$A$ is isomorphic to $C(P)$ where $P \to X$ is the principal $S^1$-bundle over $X$ associated to $\Ll$ and the gauge action comes from the principal $S^1$-action.
\end{enumerate}
\end{Prop}

\begin{Rem}
The previous hypotheses ensure that $A$ and $B$ with the gauge action form a \emph{Hopf-Galois extension}, as first introduced in \cite{HGExtensionKT} and later related to ``noncommutative principal bundle'' -- see \eg \cite{QgpGaugeThQBM,QgpGaugeThQErrBM,PpalBundleHajac,PpalActionEllwood}, among numerous others. We could deduce the previous result from these general directions of reseach, however we include the short proof below for self-containment and clarity.
\end{Rem}

\begin{proof}
Since $A$ is unital, there is a compact space $P$ such that $A = C(P)$; $X$ is of corse the spectrum of $A$. It is also clear that $1_A \in B$. Consider the characters $\varphi \colon A \to \C$ of $A$: they identify with points of $P$. Any such $\varphi \in P$ induces by restriction a character on $B$, \ie $\varphi(b) =b(x_0)$ for all $b \in B$ and some $x_0 \in X$. This defines the map $P \to X$.

We can now apply the Serre-Swan Theorem to $E := A^{(1)}$ and get a vector bundle $\Ll \to X$ whose module of sections is $E$. Since $\Ll$ is a bundle over $X$, for any $x_0 \in X$ we can define a map $\Ev_{x_0} \colon \Ll \to \Ll_{x_0}$ where $\Ll_{x_0}$ is the fiber of $\Ll$ over $x_0$ -- in fact, this map extends to a $C^*$-algebra morphism $\Ev_{x_0} \colon B \rtimes _E \Z \to \C \rtimes _{\Ll_{x_0}} \Z$, thereby defining a notion of ``local'' equality. $A$ is commutative and contains $E \otimes_B E$, hence we must have $\xi \otimes \eta = \eta \otimes \xi$ for any $\xi, \eta \in \Ll_{x_0}$. This is only possible if the rank of $\Ll$ is $1$ or $0$ -- the latter is impossible since $E$ generates $A$, unless $A = \{ 0\}$.

Thus, for any $x_0 \in X$, we can find a section $\xi_1$ of $\Ll$ such that $\langle \xi_1, \xi_1 \rangle = 1$ locally around $x_0$. Consequently $\varphi(\xi_1)^* \varphi(\xi_1) = \varphi(\langle \xi_1, \xi_1 \rangle) = 1$, which shows that $\varphi(\xi_1)$ is a complex number of module $1$. Moreover, any algebraic $F \in A$ can be written locally around $x_0$ as a sum 
$$ F = \sum_{} b_n (\xi_1)^n $$
Thus $\varphi$ is totally determined by $x_0$ and $\lambda := \varphi(\xi_1) \in U(1) \subseteq \C$. Moreover, any such choice $(x, \lambda)$ for $x$ around $x_0$ defines a character of $A$.

\smallbreak

Finally, the gauge action $\gamma$ is given by 
\begin{align*}
\varphi(\gamma_z(b)) &= b
&
\varphi(\gamma_z(\xi_1)) &= z \xi_1
\end{align*}
for any $b \in B$ and complex number $z$ with $|z|=1$, therefore completing the proof.
\end{proof}

\section{Derivations and extensions of connexions}\label{Section:Derext}

The next step of our construction is to link the spectral triple on the base algebra with the Hilbert bimodule, \textsl{via} a connexion satisfying certain conditions (see Definition \ref{Def:Main} below). For the rest of this section, we fix a Rieffel spectral triple $(\pi,D,\Hh)$ on a unital $C^*$-algebra $B$ with components $\partial_j$ and associated subspace $\Bb$ (see Definition \ref{Def:TrSp} and \ref{Def:RieffelTrSp}). We denote again by $B_1\subseteq B$ the operator-$*$-algebra closure of $\Bb$ associated to the derivation $\delta_D =[D,\;\cdot\;]$, see Proposition \ref{Prop:DefB1}.
\begin{Def}
\label{Def:Main} 
A \emph{two-sided Hermitian $D$-connexion} on a Hilbert bimodule $E$ over $B$ is a map $\nabla \colon \modHilbL \to E \otimes _B \Omega^1_D$   defined on a dense subspace $\modHilbL\subseteq E$ which is stable under both left and right multiplication by $\Bb$ such that
\begin{itemize}
\item
$\nabla $ is a Hermitian $D$-connexion in the sense of Definition \ref{Def:HermConnexion}.
\item ${}_B \langle \modHilbL, \modHilbL \rangle \subseteq B_1$ and
for all $j$ and $\xi, \eta \in \modHilbL$
\begin{equation}
\label{Eqn:CompScalProdj2}
 {}_B \langle \nabla _j(\xi), \eta \rangle + {}_B \langle \xi, \nabla _j(\eta) \rangle = \partial _j \left( {}_B \langle \xi, \eta \rangle \right),
\end{equation}
 as well as for all $j$, $\xi \in \modHilbL $ and $b \in \Bb$
\begin{equation}
\label{Eqn:DerivNabla2}
\nabla _j(b \xi) = b \nabla _j(\xi) + \partial _j(b) \xi,
\end{equation}
where $\nabla_j:\modHilbL\to E$ denote the components of $\nabla$ in the sense of Lemma \ref{Lem:DecompNabla}.

\end{itemize}
\end{Def}

\begin{Rem}
The main difference between \eqref{Eqn:CompScalProdj} and \eqref{Eqn:CompScalProdj2} is that the second uses the \emph{left} scalar product whereas the first involves the \emph{right} scalar product.
\end{Rem}

\begin{Rem}
\label{Rem:ProdScalXXOmega}
We can define in $\Omega^1_D$ the following scalar products:
\begin{align*}
{}_B \langle \nabla \xi, \eta \rangle &:= \sum {}_B \langle \nabla _j(\xi), \eta \rangle \otimes \clgen_j
&
{}_B \langle \xi, \nabla \eta \rangle &:= -\sum {}_B\langle \xi,  \nabla _j(\eta) \rangle \otimes \clgen_j.
\end{align*}
It is easy to check that $\left( {}_B \langle \nabla \xi, \eta \rangle \right)^* = {}_B \langle \eta, \nabla \xi \rangle $. Just like in Remark \ref{Rem:StabE1}, the condition \eqref{Eqn:CompScalProdj2} is equivalent to:
\begin{equation}
\label{Eqn:CompScalProd3}
{}_B \langle \nabla \xi, \eta \rangle - {}_B \langle \nabla \xi, \eta \rangle = [D, {}_B \langle \xi, \eta \rangle ].
\end{equation}
This definition of the scalar products on $\Omega^1_D$ is coherent with the pairing defined in (4.2) p. 17 of \cite{SpecFlowKL}.
\end{Rem}

\begin{Ex}
If $G$ is a given Lie group, $B$ is equipped with a $G$-action $\alpha$ and $\modHilb $ is a Hilbert module over $B$ endowed with a compatible \emph{Hilbert bimodule action $\beta$} of $G$, \ie $\beta_g(\xi b) = \beta_g(\xi) \alpha_g(b)$, $\beta_g(b \xi) =  \alpha_g(b) \beta_g(\xi)$ and
\begin{align*}
\langle \beta_g(\xi), \beta_g(\eta) \rangle _B &= \alpha_g \left( \langle \xi, \eta \big \rangle_B \right)
&
{}_B \langle \beta_g(\xi), \beta_g(\eta) \rangle &= \alpha_g \big( {}_B \langle \xi, \eta  \rangle \big),
\end{align*}
then the $G$-smooth elements $\modHilbL$ of $\modHilb$ and the infinitesimal generators $\nabla  _j$ of the action $\beta$ define a two-sided Hermitian connexion (this an easy generalisation of Proposition 4.9 in \cite{PairingsQHM}). We can make it concrete using the Dirac-Rieffel operator as in Definition \ref{Def:RieffelTrSp}.
\end{Ex}

\begin{Rem}
In the situation described in \cite{TrSpLieGpGG}, \ie if the Dirac operator arises from an \emph{ergodic} action ($\forall g \in G, \alpha_g(b) = b \implies b \in \C 1$) of a compact Lie group $G$, the results of \cite{EquivHilbModGoswami} show that all $G$-$B$-Hilbert modules are embeddable, \ie $E$ is $G$-equivariantly embedded into $B \otimes \Hh_0$ for some Hilbert space $\Hh_0$ endowed with a $G$-action.
\end{Rem}

Consider now a fixed Hilbert bimodule $E$ over $B$ and denote by $\A:=B\rtimes_E\mathbb{Z}$ the associated GCP. Suppose $\nabla$ is a two-sided Hermitian $D$-connexion on $E$.  We denote by $E_1 \subseteq E$ the operator space obtained from $E$ and $\nabla $ by Proposition \ref{Prop:OpStarMod}.
\begin{Lem}
\label{Lem:DefDeriv}
Denote by $\Aa$ the sub-$*$-algebra of $\A$ generated $*$-algebraically from $E_1$ and $B_1$. The components $\partial _j$ and $\nabla _j$ of $D$ and $\nabla $, respectively, extend uniquely into  $*$-derivations $\nablat_j$ defined on $\Aa$ with values in $\A$. These extensions $\nablat_j$ preserve the gauge-action on $\Aa$.
\end{Lem}

\begin{proof}
If such derivations exist, then given any $*$-algebraic combination $\Xi$ of elements $\xi \in E_1$ and $b \in B_1$ (\eg $\Xi= \xi_1 b_1 \xi_2^* b_2$), the $*$-derivation property gives a unique expression for $\nablat_j(\Xi)$. On our example, we get:
$$
 \nablat_j( \Xi) = \nablat_j(\xi_1) b_1 \xi_2^* b_2 + \xi_1 \partial _j(b_1) \xi_2^* b_2 +  \xi_1 b_1 \nablat_j(\xi_2)^* b_2 +  \xi_1 b_1 \xi_2^* \partial _j(b_2).
$$
This proves unicity. We now need to show that these expressions indeed define a derivation on $\Aa$. To this end, it suffices to check that the analogs of relations (i) to (iv) of Definition \ref{Def:RepHilbBimod}  behave suitably under the derivation. This is in turn ensured by the assumptions and Definition \ref{Def:Main}. For instance, regarding (i):
$$
\nablat_j \big( \xi^* \zeta \big) = \partial _j \big( \langle \xi, \zeta \rangle_B \big) = \langle \nabla_j( \xi), \zeta \rangle_B  + \langle  \xi, \nabla_j(\zeta) \rangle_B = \nablat_j( \xi)^* \zeta + \xi^* \nablat_j(\zeta),
$$
where we used \eqref{Eqn:CompScalProdj}. Regarding  (ii), we have:
$$ 
\nablat_j(\xi b) = \nablat_j(\xi) b + \xi \partial _j(b)
$$
by Equation \eqref{Eqn:CompScalProdj2}.

\medbreak

For a formal proof, decompose $\Aa$ into an algebraic sum of spectral subspaces (see equation \eqref{Eqn:SpecSubspace}) and use induction on the degree to prove first that $\nablat_j$ is well-defined, second that it is compatible with products. 
\end{proof}

\begin{Def}\label{Def:canonicalrep} Let $B$ be a $C^*$-algebra, $E$ a Hilbert bimodule over $B$ and $\A:=B\rtimes_E\mathbb{Z}$.
We define a $B$-Hilbert module $X$ as the $C^*$-module-completion of $\A$ for the $B$-valued scalar product:
$$ \langle \Xi, \Xi' \rangle_B := \E( \Xi^* \Xi' ) ,$$
where the conditional expectation $\E \colon \A \to B$ is induced by the gauge action. 
\end{Def}
\begin{Rem}
\label{Rem:InclA}
That is, $\E$ is the application defined by
$$\E(a):=\int_{S^1}\lambda.a\; d\lambda.$$
Equivalently, $X$ is the Hilbert sum obtained as direct sum of the spectral subspaces $\A^{(k)}$, $k \in \Z$ (as defined equation in \eqref{Eqn:SpecSubspace}).
Since the conditional expectation $\E$ is faithful on $\A$, we have an injection $\A \inj X$. Moreover, the conditional expectation is norm-contracting, therefore this injection is continuous.
\end{Rem}

We suppose that the Rieffel spectral triple $(\Hh,D)$ on $B$ is associated to a Clifford module $S$ over $\mathbbm{C}l(n)$ and again denote the images of the generators of $\mathbbm{C}l(n)$ in $\mathbbm{B}(S)$ by $\gamma_i$. Let $\Hh=\Hh_0\otimes S$ be the underlying Hilbert space. Suppose furthermore that $\nabla$ is a two-sided Hermitian $D$-connexion on $E$.

\begin{Prop}
\label{Prop:ExtConn}
Under the above hypothese, the derivations $\nablat_j \colon \Aa \to A$ from Lemma \ref{Lem:DefDeriv} combine to form a Hermitian $D$-connexion $\nablat \colon \Aa \to X \otimes _B \bo(\Hh)$ on setting
$$ \nablat(\Xi) := \sum_{j=1}^n \nablat_j(\Xi) \otimes \id_{\Hh_0} \otimes \clgen_j .$$
If $E$ is finitely generated and projective both as a left- and right-$B$-module, then 
\begin{enumerate}[(i)]
\item
$X$ admits a ``frame'' in the sense of Definition \ref{Def:FrameInf},
\item
$\nablat$ is closable and the $B_1$-module $X_1$ obtained from $\Aa$ by Proposition \ref{Prop:OpStarMod} is an operator $*$-module over $B_1$. It is dense in $X$ and the inclusion $X_1 \inj X$ is completely bounded.
\end{enumerate}
\end{Prop}
\begin{proof}
We know that the $\nablat_j$ are well-defined maps from $\Aa$ to $A\subseteq X \otimes B(S)$ -- see Lemma \ref{Lem:DefDeriv}.


We prove first that $\nablat$ defines a $D$-connexion (Definition \ref{Def:OperatorConnexion}). To this end, we calculate for all $\Xi\in\Aa$ and $b\in \Bb$
\begin{multline*}
\nablat(\Xi b) = \sum \nablat_j(\Xi b) \otimes \id_{\Hh_0} \otimes \clgen_j = \sum_{} ( \nablat_j(\Xi) b + \Xi \, \nablat_j(b) ) \otimes \id_{\Hh_0} \otimes \clgen_j \\
= \nablat(\Xi) b + \Xi \otimes \left(\sum_{} \partial _j(b) \otimes \clgen_j \right) = \nablat(\Xi) b + \Xi \otimes [D, b].
\end{multline*}
We now evaluate $[D, \langle \Xi_1, \Xi_2 \rangle ]$ for $\Xi_1, \Xi_2$ in $\Aa$ in order to prove that $\nablat$ is Hermitian (Definition \ref{Def:HermConnexion}). Without loss of generality, we can assume that $\Xi_1$ and $\Xi_2$ are gauge homogeneous and of same degree -- otherwise, the scalar product vanishes and since $\nablat$ preserves the gauge action, the equality holds. We then have:
\begin{multline*}
[D, \langle \Xi_1, \Xi_2 \rangle ] = \sum_{} \partial _j \left( \langle \Xi_1, \Xi_2 \rangle  \right) \otimes \clgen_j = \sum_{} \partial _j \left( \Xi_1^* \Xi_2 \right) \otimes \clgen_j\\
= \sum_{} \left( \nablat_j(\Xi_1)^* \Xi_2 + \Xi_1^* \nablat_j(\Xi_2) \right) \otimes \clgen_j \\
= \sum_{} \langle \Xi_1, \nablat_j(\Xi_2) \rangle \cdot \id_{\Hh_0} \otimes \clgen_j - \Big( \sum_{} \langle \Xi_2, \nablat_j(\Xi_1) \rangle \cdot \id_{\Hh_0} \otimes \clgen_j \Big)^*\\
= \langle \Xi_1, \nablat(\Xi_2) \rangle  - \langle \nablat(\Xi_1),\Xi_2 \rangle,
\end{multline*}
where we used the derivations $\nablat_j$ on $\Aa$, together with the definition of the scalar product between elements of $X$ and $X \otimes _B \bo(\Hh)$ given in Remark \ref{Rem:ProdScalXXOmega}.

\medbreak

To prove (i), remark that given a right $B$-module $F$ and a $B$-bimodule $F'$, it is easy to show that if $(\xi_j)_{j=1}^n$ and $(\zeta_k)_{k=1}^m$ are right frames for $F$ and $F'$, then $(\xi_j \otimes \zeta_k)$ is a right frame for $F \otimes _{B} F'$. Consequently, if $E$ is finitely generated and projective as a right-module, all $\A^{(k)}$ for $k \in \N$ admit a frame whose elements are in $\Aa$. The same argument applies to $\A^{(-k)}$ for $k \in \N$, provided that $E$ be f.g. projective as \emph{left} $B$-module. We denote $(\xi_j)_{j=-\infty }^\infty $ the elements thus obtained, organised in blocks of elements of $\A^{(k)}$ for increasing $k$. We now prove that 
$$ T_N := \sum_{k = -N}^N \xi_k \langle \xi_k, \cdot \rangle $$
satisfies $T_N \Xi \to \Xi$ for all $\Xi \in X$.

\medbreak

It is clear from the definition that any $\Xi \in X$ is $\varepsilon$-close to a $\Xi_0 \in \Aa$ (where $\Xi_0$ contains only elements of a finite number of $\A^{(k)}$ involved, \ie an element of the algebraic direct sum). For a fixed $\Xi_0$, we clearly have $T_N \Xi_0 = \Xi_0$ for $N$ large enough. If we can prove $\| T_N \| \leqslant 1$ for all $N$, we are done. From the definition of $X$ and because $T_N$ preserves the degre, $\| T_N \| = \sup \big\{ \| T_N \restriction \A^{(k)} \| \big| k \in \Z \big\}$. But for any fixed $k$, $T_N \restriction \A^{(k)}$ is an increasing finite sum of positive operators (in $\bo_B(\A^{(k)})$). For $N$ large enough this sum is the identity. Thus for $N' \leqslant N$ we have 
$$
0 \leqslant (T_{N'} \restriction \A^{(k)}) \leqslant (T_N \restriction \A^{(k)}) = 1,
$$
which implies that $\| T_{N'} \| \leqslant 1$ and proves that $T_N$ is a ``frame'' in the sense of Definition \ref{Def:FrameInf}.

\medbreak

We can then apply Proposition \ref{Prop:FrameClosed} and prove that $\nablat$ is closable. The hypotheses of Proposition \ref{Prop:OpStarMod} are thus satisfied for each of the modules $\A^{(k)}$, hence all $\A^{(k)}$ for $k \in \Z$ are operator $*$-modules. Denote $P_k$ the associated matrix in $M(B)$ -- following the notations of \cite{SpecFlowKL}. The operator $*$-module $X_1$ is then obtained by Corollary \ref{Cor:FrameE1}. The associated projector is the direct sum $\bigoplus_{k \in \Z} P_k$.
\end{proof}

We note the following:
\begin{Cor}
Combining the derivations $\nablat_j$ of Lemma \ref{Lem:DefDeriv}, we obtain on $\Aa$ a closable derivation defined by
$$
\delta(a) = \sum_{j=1}^n\nablat_j(a) \otimes \clgen_j 
$$
with values in $X \otimes_B \bo(\Hh)$.
\end{Cor}

\begin{proof}
We consider a sequence $a_n \in \Aa$ such that $a_n \to a$ and $\delta(a_n) \to y$. By Remark \ref{Rem:InclA}, this implies that $a_n$ and $\delta(a_n)$ seen as sequences in $X$ and  $X \otimes _B \bo(\Hh)$ are convergent. But by Proposition \ref{Prop:ExtConn}, $\nablat$ is closable and therefore $y = 0$, which completes the proof.
\end{proof}

\begin{Def}
\label{Def:A1}
Let $\A_1 \subseteq \A$ be the operator $*$-algebra closure of $\Aa$ with respect to $\delta$ (Proposition \ref{Prop:DefB1}). It is therefore a subalgebra closed under holomorphic functional calculus.
\end{Def}
%
%
\section{Conditional expectation, vertical class and main results}\label{Section:TheLifting}
Let $B$ be a $C^*$-algebra and $E$ a Hilbert bimodule over $B$. Denote $A:=B\rtimes_E\mathbbm{Z}$ the (generalized) crossed product of $B$ by $E$. 
\begin{Def}
\label{Def:verticalclass}
We will call \emph{vertical class} the odd Kasparov $(A,B)$-module defined by the canonical representation of $A$ on the $A$-$B$ $C^*$-correspondence $X=\bigoplus_{k\in\mathbbm{Z}} A^{(k)}$ of Definition \ref{Def:canonicalrep} by left multiplication, and the operator $D_v$ determined by
$$(D_v\xi)_n=n\xi_n.$$
\end{Def}

\begin{Rem}
Note that $E\cong A^{(1)}$, and if $E$ is finitely generated projective then $E^*\cong\bar E\cong A^{(-1)}$ . Thus $X$ corresponds to the module
$$\mathcal{E}_\infty=\bigoplus_{n\in\mathbbm{Z}} E_\infty^{\otimes n}$$ introduced by Pimsner on page 195 of the fundamental article \cite{PiCrG}. Pimsner starts with a general $A$-$A$-$C^*$-correspondence $E$ which necessitates the construction of a Hilbert bimodule over the larger algebra $\mathcal{F}_E$ -- our situation is simpler in as far as we start out with a Hilbert bimodule.
\end{Rem}

Let now $B$  be equipped with a Rieffel-spectral triple $(\Hh,D_h)$ and operator-$*$-algebra  $B_1$. Further let $\nabla\colon \mathcal{E}\to E\otimes_B \bo(\Hh)$ be a two-sided $D_h$-connexion, $\underline\nabla$ the extension to a connexion on $X$, and $X_1$ the associated operator-$*$-module closure constructed in Proposition \ref{Prop:ExtConn}.

We are now ready to prove the following: 
\begin{Prop}
\label{Prop:CorrUnboundedKmod}
The pair $(X_1, \nablat)$ defines a correspondence from $(X,D_v)$ to $(\Hh, D_h)$ (compare 6.3 of \cite{SpecFlowKL}) if we use the natural $B_1$-operator module structure on $X_1$.
\end{Prop}

\begin{proof}
We assumed that $B$ is unital, thus it is clear that the unbounded Kasparov module $(\Hh, D_h)$ is essential. Checking two out of four points of Definition 
of correspondence is straightforward:
\begin{enumerate}[(1)]
\item
Proposition \ref{Prop:ExtConn}, (ii) ensures that $X_1$ is an operator $*$-module and Corollary \ref{Cor:FrameE1} ensures that the map $\nablat $ is a completely bounded Hermitian $D_h$-connexion from $X_1$ to $X \otimes _B \bo(\Hh)$.
\item
By definition, any $b \in B_1$ lies in the domain of the closure of $\delta_{D_h}$. Thus  by Lemma \ref{Lem:CloseDeriv} we have $b (\dom(D_h)) \subseteq \dom(D_h)$ and $[D_h, \cdot ] \colon B_1 \to \bo(\Hh)$ is completely bounded.
\end{enumerate}
Before treating the last two points, we establish that $\Cc := \Aa \odot_{B_1} \dom(D_h)$ (algebraic tensor product) is a core for $1 \otimes _{\nablat} D_h$. Recall that the algebra $\Aa$ used in the definition of $\Cc$ was introduced in Lemma \ref{Lem:DefDeriv}.

Consider the selfadjoint operator $\Diag(D_h) = D_h \otimes \id$ acting on the Hilbert space $\Hh \otimes \ell^2(\Z)$. The algebraic tensor product of $\dim(D_h)$ with vectors of compact support in $\ell^2(\Z)$ is a core $\Cc'$ of $\Diag(D_h)$ by definition. Upon identifying $X$ with $P B_1^\infty $, the operator $1 \otimes _{\nabla ^\text{Gr}} D_h$ coincides with the compression $P \Diag(D_h) P$ whose domain is $P \dom(\Diag(D_h))$. It is readily checked that $P \Cc'$ is then a core for this compression.

The operator $1 \otimes _{\nablat} D_h$ is just a bounded perturbation of $1 \otimes _{\nabla ^\text{Gr}} D_h$ (see Theorem 5.4 in \cite{SpecFlowKL}) and they therefore have the same domain. Even better, they share the same cores. Hence, it only remains to identify:
$$ 
\Aa \odot _{B_1} \dom(D) \simeq P B_1^{\text{alg}, \infty} \odot _{B_1} \dom(D) \simeq P \dom(D)^{\text{alg}, \infty } \simeq P \Cc' 
$$
which holds because $1 \in B_1$ and $B_1 \odot _{B_1} \dom(D) = \dom(D)$. Here we denote by $B^{\text{alg}, \infty }$ the sequences of finite support in $B^\infty$ which were denoted $c_0(B)$ in \cite{SpecFlowKL} Definition 3.3. Hence $\Cc = \Aa \odot _{B_1} \dom(D_h)$ is indeed a core for $1 \otimes _{\nablat} D_h$. We can now proceed with the remaining two points:
\begin{enumerate}[(1)]
\setcounter{enumi}{2}
\item
As mentionned in \cite{SpecFlowKL} below Definition 6.3, given the core $\Cc$ for $1 \otimes _{\nablat} D$, it suffices to prove that 
\begin{itemize}
\item
any $a$ in the subalgebra $\Aa$ sends $\Cc$ to itself,
\item
$[1 \otimes _{\nablat} D, a]$ defined from $\Cc$ to $X \otimes _B \Hh$ extends to a bounded operator on $X \otimes _B \Hh$.
\end{itemize}
The definition of $\Cc$ ensures that $\forall a \in \Aa, a \Cc \subseteq \Cc$. Moreover on simple tensors in $\Cc$, the operator $1 \otimes _{\nablat} D$ acts by (compare Lemma 5.1 p.21 of \cite{SpecFlowKL}):
$$ (1 \otimes _{\nablat} D)(x \otimes \xi) = x \otimes D(\xi) + \nablat(x) \xi, $$
where $\nablat(x) \in X \otimes_B \bo(\Hh)$ acts naturally on $\xi \in \Hh$. We can then calculate:
\begin{multline*}
[1\otimes _{\nablat} D, a] (x \otimes \xi) =  (1 \otimes _{\nablat} D)(a x \otimes \xi) -  a (1 \otimes _{\nablat} D)(x \otimes \xi) \\
= a x \otimes D(\xi) + \nablat(a x) \xi - a \big( x \otimes D(\xi) + \nablat(x) \xi \big) \\
= \big(\nablat(a x) - a \nablat(x) \big)\xi = \nablat(a) x \otimes \xi,
\end{multline*}
since the last part of the computation takes place in $\Aa$, where we know that $\nablat$ is a derivation. For $a \in \Aa$,  the operator $\nablat(a)$ is clearly the restriction of a bounded operator on $X\otimes_B H$.
\item
This last condition is an abbreviation for points (1) and (2) of Theorem 1.3 in \cite{SpecFlowKL}. On the ``algebraic elements'' of $\Cc$, the required stability properties are clear. The extension of $[\partial _t, 1 \otimes _{\nablat} D] (\partial _t - i \mu)^{-1} \colon \Cc \to X \otimes _B \Hh$ to a bounded operator is clear, since $[\partial _t, 1 \otimes _{\nablat} D](x \otimes \xi) = 0$ for any simple tensor of $\Cc$.
\end{enumerate}
\end{proof}
We can now state our main result:
\begin{Th}\label{Thm:Main} Let $E$ be a Hilbert bimodule over a $C^*$-algebra $B$, which is finitely generated as both left- and right-module. Suppose that $B$ is unital and equiped with an odd Rieffel spectral triple $(\Hh,D_h)$. 

If $\nabla \colon \modHilbL \to E \otimes_B \Omega^1_{D_h}$ is a two-sided Hermitian  $D_h$-connexion on $E$  and $\nablat$ the associated $D_h$-connexion on $X$ according to Proposition \ref{Prop:ExtConn}, then the operator
\[\begin{pmatrix} 0&D_v\otimes 1-i\; 1\otimes_{\underline{\nabla}} D_h\\D_v\otimes 1+i\; 1\otimes_{\underline{\nabla}} D_h&0\end{pmatrix}\]
on $X\otimes_B\Hh\oplus X\otimes_B\Hh$ equipped with the standard grading defines a spectral triple on $A=B\rtimes_E \mathbbm{Z}$ with respect to the natural action on the left by  multiplication.

If $(\pi,\Hh,D_h)$ is an even Rieffel spectral triple with grading operator $\gamma$, then 
$$(X\otimes_B\Hh,D_v\otimes \gamma+1\otimes_{\nablat}D_h)$$
is a spectral triple on $A=B\rtimes_E \mathbbm{Z}$.

In both cases, this spectral triple is a representative of the Kasparov product of the vertical and horizontal class.
\end{Th}
\begin{proof}
This is now a consequence of Theorem \ref{Theorem:products} from the appendix and Proposition \ref{Prop:CorrUnboundedKmod}.
\end{proof}

\section{Examples}
\label{Sec:Examples}

Our first examples are Quantum Heisenberg Manifolds (QHM), which were introduced by Rieffel in \cite{RieffelDefQuant} as one of the original examples of Rieffel deformations. They were subsequently studied among others by Abadie (see \eg \cite{AbadieFixedPts,AbadieEE}), Chakraborty (see \cite{GeomQHM,MetricQHeisenMan}) and ourselves \cite{PairingsQHM, ChernGrensingGabriel} -- for a more complete survey of literature about QHM, see the introduction of \cite{PairingsQHM}. Quantum Heisenberg manifolds form a family of $C^*$-algebras $\QHM[c][\mu,\nu]$ indexed by $c \in \Z$ and $\mu,\nu \in \R$.

Here, we follow Definition 1 of \cite{GeomQHM} to introduce the QHM. Given $c \in \Z$, $\mu, \nu \in \R$, $\QHM[c][\mu, \nu]$ is the envelopping $C^*$-algebra of the $*$-algebra generated by
$$
 \big\{ F \in C_\text{c}(\Z \to C_\text{b}(\R \times S^1)) \big| F(x + 1, y, p) = e(c p y) F(x,y,p) \big\} 
$$
equipped with multiplication:
\begin{equation}
\label{Eqn:Comp}
 (F_1 \cdot F_2)(x,y,p) = \sum_{q \in \Z} F_1(x -(q-p) \mu ,y -(q-p) \nu,q) F_2(x - q \mu,y- q \nu, p-q)
\end{equation}
and involution:
$$F^*(x,y,p) = \overline{F(x,y,-p)} .$$ 
This $C^*$-algebra carries a pointwise continuous action $\alpha$ of the \emph{Heisenberg group} $ \HH$, \ie the subgroup of $GL_3(\R)$ of the matrices
\begin{equation}
\label{Eqn:ParamHH}
\begin{pmatrix}
1 & s & t \\
0 & 1 & r \\
0 & 0 & 1
\end{pmatrix}
, r,s,t \in \R.
\end{equation}
The explicit expression of the action $\alpha$ is (compare \cite{GeomQHM}, equation (2.4)):
$$ 
\alpha_{(r,s,t)}(F)(x,y,p) = e(p (t + c s(x-r))) F(x-r, y-s, p)
$$
where $x,y \in \R$ and $p \in \Z$. Considering only $\alpha_{(0,0,t)}$ actually yields an action of $S^1 = \R/(2 \pi) \Z$, which we consider as the gauge action for our situation. This definition leads to the spectral subspaces:
$$ \QHM[c][\mu,\nu]^{(k)} = \{ F \in \QHM[c][\mu,\nu] \big| F(x, y, p) = \delta_{k,p} F(x,y,k) \big\} \} ,$$
where $k \in \Z$ and $\delta_{k,p}$ is $1$ if $k=p$ and $0$ else. In particular, the gauge-invariant elements of $\QHM[c][\mu,\nu]$ are just the algebra $B = C(T^2)$. It is easily seen that $\QHM[c][\mu,\nu]^{(1)}$ is a Hilbert bimodule over $B$: the left and right actions are restrictions of the product on $\QHM[c][\mu,\nu]$, while the left and right $B$-valued scalar products are $\langle \xi, \eta \rangle_B := \xi^* \eta$ and ${}_B \langle \xi, \eta \rangle = \xi \eta^*$.

We denote by $\E \colon \QHM[c][\mu,\nu] \to B$ the conditional expectation induced by the gauge action, \ie 
$$ \E(F) := \int_{S^1} \alpha_{(0,0, t)}(F) dt .$$
There is a unique $T^2$-invariant and normalised trace on $B$, which is 
$$
\tau_0(b) = \int_{T^2} b(x,y) d \lambda(x,y)
$$ 
where $\lambda$ is the Lebesgue measure. Composing these two maps, we obtain a $\HH$-invariant trace $\tau$ on $\QHM[c][\mu,\nu]$ (compare \cite{GeomQHM}, p.427):
$$ \tau(F) := \tau_0 \left( \E(F) \right) = \int_{T^2} F(x,y, 0) dx dy .$$
Let us now prove that $\QHM[c][\mu,\nu]$ is a GCP using Theorem \ref{Thm:DefGCP} and the following Proposition 2.6 of \cite{PairingsQHM}:
\begin{Lem}
\label{Lem:Frame}
There are elements $\xi_1, \xi_2 \in \QHM[c][\mu,\nu]^{(1)}$ with $\xi_1^* \xi_1 + \xi_2^* \xi_2 = 1$. 
\end{Lem}

\begin{proof}
For self-containment, we sketch the proof. We can find two real-valued, $1$-periodic functions $\chi_1, \chi_2$ on $\R$ such that $\chi_1^2(x) + \chi_2^2(x) = 1$ while $\chi_1(x) = 1$ for $x \in [-1/6, 1/6]$ and $\chi_2(x) = 1$ for $x \in [ 1/3, 2/3]$. We can then define $\xi_1$ in $\QHM[c][\mu,\nu]$ by $\xi_1(x,y,p) = \delta_{1, p} \chi_1(x)$ for $x \in [-1/2, 1/2]$, which we extend into elements of $\QHM[c][\mu,\nu]$. Applying the same construction to $\chi_2$ on $[0,1]$, we obtain $\xi_2$. It is easy to check using suitable covers of $T^2$ that $\xi_1^* \xi_1 + \xi_2^* \xi_2 = 1$.
\end{proof}

On the gauge-invariant subalgebra $B  \simeq C(T^2) \subseteq \QHM[c][\mu,\nu]$, the action of $\HH$ factors through an action of $G = T^2$ -- which is just the canonical action of $T^2$ on $B$. The usual differential structure on $T^2$ corresponds to $\Bb := C^\infty (T^2)$. It can equivalently be obtained from the $T^2$-smooth elements of $B$. The natural spectral triple structure on $B$ can be obtained for instance by the method of \cite{TrSpLieGpGG}. We thus recover an unbounded operator
$$ D := \partial_1 \otimes \gamma_1 + \partial _2 \otimes \gamma_2 ,$$
which is acting on $L^2(T^2) \otimes \C^2$. The left action of $\Bb$ on $L^2(T^2)$ is provided by pointwise multiplication. The matrices $\gamma_1, \gamma_2$ satisfy $\gamma_j \gamma_k + \gamma_k \gamma_j = 2 \delta_{jk}$ and $\gamma_j^* = - \gamma_j$. The spectral triple thus obtained is even with grading operator $\gamma_3 = i \gamma_1 \gamma_2$ which commutes with both $\gamma_1$ and $\gamma_2$.

Let $\partial_1, \partial_2, \partial_3 \in \HHLie$ in the Lie algebra of $\HH$ be the infinitesimal generators associated to the parametrisation \eqref{Eqn:ParamHH}. It is easy to see that they satisfy the commutation relations:
\begin{align*}
[\partial _j, \partial _3] &= 0
&
[\partial _1, \partial _2] &= \partial _3
\end{align*}
for $j = 1, 2$. We introduce a smooth version of $M^c_{\mu, \nu}$ by setting:
\begin{equation}
\label{Def:ModLisse}
\modQHMl[c][\mu, \nu] := \left\{ \xi \in M^c_{\mu, \nu}: (r,s,t) \mapsto \alpha_{(r,s,t)}(\xi) \text{ is in }C^\infty (\HH, \QHM[c][\mu,\nu]) \right\} .
\end{equation}
It follows immediately from the definition that $\HHLie$ acts by derivations on $\modQHMl[c][\mu, \nu]$. We will write $\partial _j \xi$ for the action of the infinitesimal generator $\partial _j$ on this ``smooth module''.
\begin{Prop}
\label{Prop:Nabla}
A connexion $\nabla $ on $\modQHMl[c][\mu, \nu]$ associated to the canonical spectral triple on $B$ is defined by
$$ \nabla (\xi) = \partial _1(\xi) \otimes \gamma_1 + \partial _2(\xi) \otimes \gamma_2 .$$
In other words, the ``components'' of $\nabla$ defined in Lemma \ref{Lem:DecompNabla} are $\partial _j$. Moreover, $\nabla$ is a two-sided Hermitian connexion in the sense of Definition \ref{Def:Main} and closable.
\end{Prop}

\begin{proof}
This connexion is associated to the spectral triple induced by the unbounded operator
$$
D = \partial _1 \otimes \gamma_1 + \partial _2 \otimes \gamma_2.
$$
The commutators of $D$ with elements $b$ of $B$ is $[D, b] = \partial_1 ^{\Bb}(b) \otimes \gamma_1 + \partial _2^{\Bb}(b) \otimes \gamma_2$. We first prove that $\nabla $ is indeed a connexion in the sense of Definition \ref{Def:Conn}. Since $\alpha$ is pointwise continuous, $\modQHMl[c][\mu, \nu]$ (also denoted $\modQHMl$ for brevity) is dense in $M^c_{\mu,\nu}$. Moreover, it is clear from \eqref{Def:ModLisse} that $\modQHMl$ is stable under left and right multiplication by elements of $\Bb$. Let $\xi \in \modQHMl$ and $b \in \Bb$, then
$$
\nabla_j (\xi b) = \nabla_j (\xi) b +\xi \partial _j(b),
$$
since the action of $\HH$ restricts to that of $T^2$ on $B$ \ie $\partial_j ^{\QHMl[c][\mu,\nu]}(b) = \partial _j^{\Bb}(b)$ for any $b \in \Bb$ and $j = 1, 2$. This proves that the equation \eqref{Eqn:DerivNabla} is satisfied and thus that $\nabla $ is a connexion (\textsl{per} Lemma \ref{Lem:DecompNabla}). The same argument applies to $\partial _j^{\Bb}(b \xi)$, thereby proving that $\nabla $ is also a left-connexion. To prove Hermicity, we rely on the fact that $\partial _j$ are derivations on $\QHM[c][\mu,\nu]$. First of all, the scalar products $ \langle \modQHMl, \modQHMl \rangle_B$ and ${}_B \langle \modQHMl, \modQHMl \rangle$ are both contained in $\Bb$. We then have:
\begin{align*}
 \partial ^{\Bb}_j(\xi^* \eta) &= \xi^* \partial _j(\eta) + \partial _j(\xi)^* \eta
&
 \partial ^{\Bb}_j(\xi \eta^*) &= \xi \partial _j(\eta^*) + \partial _j(\xi) \eta^*.
\end{align*}
This proves that $\nabla $ is indeed a two-sided Hermitian connexion in the sense of Definition \ref{Def:Main}. It remains to show that $\nabla $ is closable, but this is an immediate consequence of Proposition \ref{Prop:FrameClosed} and Lemma \ref{Lem:Frame}.
\end{proof}

\begin{Rem}
The previous argument applies more generally to generalized crossed products equipped with a Lie group action such that the restriction of this action to the basis algebra $B$ actually induces the spectral triple structure on $B$ \eg ergodic actions as in \cite{TrSpLieGpGG} Theorem 5.4.
\end{Rem}

\begin{Prop}
The spectral triple we obtain by applying our Theorem \ref{Thm:Main} to the connexion defined in Proposition \ref{Prop:Nabla} coincide with the one constructed in \cite{GeomQHM} Theorem 10.
\end{Prop}

\begin{proof}
The Hilbert space which we obtain from the Kasparov product is $X \otimes _B L^2(T^2)$, where $X$ is the $B$-Hilbert module induced from $\QHM[c][\mu,\nu]$ by using the conditional expectation $\E \colon \QHM[c][\mu,\nu] \to B$. The Hilbert space $L^2(T^2)$ arises from the GNS construction on $B$ with the trace $\tau_0(f) = \int_{T^2} f(x) d \lambda(x)$. 

Therefore, $X \otimes_B L^2(T^2)$ is the completion of $F \otimes b$ for $F \in \QHM[c][\mu,\nu]$ and $b \in B$ for the scalar product:
$$ \langle F_1 \otimes b_1, F_2 \otimes b_2 \rangle = \langle b_1, \E( F_1^* F_2) b_2 \rangle = \tau_0 \left( b_1^* \E(F_1^* F_2) b_2 \right) = \tau\big( (F_1 b_1)^* F_2 b_2 \big) .$$
In other words, $X \otimes _B L^2(T^2) \simeq \GNS( \QHM[c][\mu,\nu], \tau)$ where $\tau$ is the previously introduced normalised $\HH$-invariant trace on $\QHM[c][\mu,\nu]$. Thus, the Hilbert spaces and the associated representations of $\QHM[c][\mu,\nu]$ appearing in \cite{GeomQHM} Theorem 10 and by application of our Theorem \ref{Thm:Main} are the same.
\end{proof}

\section{Appendix}
\label{Section:Appendix}

\subsection{Preliminaries}
We briefly recall some basics on $KK$-theory (see the original paper \cite{MR715325} for more details). We restrict to the case of ungraded algebras.
\begin{Def} 
An \emph{even unbounded  $(A,B)$-Kasparov module} is a pair $(E,D)$ such that $E$ is a graded $B$-Hilbert module carrying an action by bounded even operators of $A$ on the right and $D$ an unbounded regular odd operator on $E$ such that
\begin{itemize}
\item $a(1+D^2)^{-1}$ extends to a compact operator on $E$ for all $a\in A$
\item the set of $a\in A$ that preserve the domain of $D$ and such that $[a,D]$ is bounded is dense in $A$.
\end{itemize}

An \emph{odd unbounded $(A,B)$-Kasparov module} is given by an even Kasparov $(A,B\otimes\mathbbm{C} l_1)$-module.
\end{Def}
It was shown by Baaj and Julg in \cite{MR715325} that there is a natural map
$$
D\mapsto q(D):=\frac{D}{\sqrt{1+D^2}}
$$
 associating to each such unbounded Kasparov module a bounded Kasparov module (which becomes a bijection on homotopy classes). 

As for bounded Kasparov modules, the odd unbounded Kasparov module have a particularly simple description, referred to as the \emph{Fredholm picture} in the bounded setting in \cite{BlackK}, Section 17.5. An \emph{unbounded odd Fredhom module} is, by definition, given by a pair $(E,D)$, where $E$ is a (trivially graded) Hilbert $B$-module carrying an action by bounded operators of $A$ on the left and $D$ an unbounded operator on $E$ such that
\begin{itemize}
\item $a(1+D^2)^{-1}$ is compact for all $a\in A$;
\item the set of all $a\in A$ that preserve the domain of $D$ and such that $[a,D]$ is bounded is dense in $A$. 
\end{itemize}
 Using a stabilization, it is easy to see that every unbounded odd Kasparov module $x$ corresponds yields canonically an unbounded Fredholm module $\hat x$ (compare Section 17.5 in Blackadar).

On the other hand, an unbounded Fredholm module $(E,D)$ yields an unbounded $(A\otimes\mathbbm{C} l_1,B)$-Kasparov module $\check x$ by setting 
$$\check x:=(E\otimes \mathbbm{C}l_1, D\otimes c)$$
where $E\otimes \mathbbm{C}l_1$ is viewed as a $B$-module and $c$ is the canonical self-adjoint generator of the Clifford algebra (as an algebra) acting by multiplication. We will refer to $\check x$ as the \emph{left hand Fredholm picture}, in order to distinguish it from $(E\otimes \mathbbm{C}l_1, D\otimes c)$ as an $(A,B\otimes \mathbbm{C}l_1)$-Kasparov module (\emph{the right hand Fredholm picture}). It is easy to see that the class of $\check x$ coincides with the class of $(E\otimes\mathbbm{C}l_1,D\otimes c)$ viewed as an $(A,B\otimes \mathbbm{C}l_1)$-Kasparov module (using Morita invariance of $KK$-theory and formal Bott-periodicity).

\subsection{The Unbounded Kasparov Product}
In this subsection, we will use the notational conventions from \cite{SpecFlowKL}.\newcommand{\mL}{\mathcal{L}}

Throughout, we fix two unbounded Kasparov modules: $(X,D_1)$ over $(A,B)$ and $(Y,D_2)$ over $(B,C)$ as in \cite{SpecFlowKL}. We denote by $\gamma_1$ and $\gamma_2$, respectively ,the (possibly trivial) grading operators on $X$ and $Y$, \ie in case that the Kasparov module is odd, we assume that it is given in the Fredholm picture. We assume throughout that $(Y,D_2)$ is essential, \ie $BY$ is dense in $Y$ and $B\Omega_{D_2}^1$ is dense in $\Omega_{D_2}^1$ (see Definition 6.2 in \locCit). We further assume that the Conventions 4.1 and 4.2 of \locCit are satisfied.

We are interested in the case where $(X,D_1)$ is odd and $(Y,D_2)$ even or odd. 

Convention 4.2 from \locCit  is in force \textsl{verbatim} -- thus as we started out with a (possibly even) Kasparov module $(Y,D_2)$, the representation of $B$ on $Y$ is by even operators and $D_2$ is an odd unbounded operator.

Furthermore, all tensor products are supposed to be equipped with the grading corresponding to the grading operator $\gamma_1\otimes\gamma_2$. The contraction map 
$$c\colon(X\hat\otimes_B\mL(Y))\otimes Y\to X\hat\otimes_B Y,\; x\otimes T\otimes y\mapsto x\otimes Ty$$ 
is then even. 

Let now $\nabla$ be a $D_2$-connexion on $X_1$. Just as in Section 5.1. of \locCit we may now define on $X\hat\otimes_A Y$ the unbounded operator $1\otimes_\nabla D_2$ with domain $\dom(\Diag(D_2))\cap X\hat\otimes Y$ by
$$
(1\otimes_\nabla D_2)(x\otimes y)=x\otimes D_2(y)+c(\nabla)(x\otimes y).
$$
Note that if $(Y,D_2)$ is even  then this defines an odd operator only if the connexion is an odd map! We will only consider such connexions. In any case, it follows from Theorem 5.4 in \locCit that this is a regular selfadjoint operator if the connexion is Hermitian.

\begin{Def}
\label{Definition:Product}
Let $\nabla\colon X_1\to X\hat\otimes_B {Y}$ be  a completely bounded $D_2$-connexion. The definition of the operator $D_1\times_\nabla D_2$ depends on the parities of $(X, D_1)$ and $(Y, D_2)$ and denotes the following operator:
\begin{enumerate}
\item If $(X,D_1)$ and $(Y,D_2)$ are odd,
\[\begin{pmatrix} 0&D_1\otimes 1-i\; 1\otimes_\nabla D_2\\D_1\otimes 1+i\; 1\otimes_\nabla D_2&0\end{pmatrix}\]
on $(X\hat\otimes_B Y)\oplus (X\hat\otimes_B Y)$ equipped with the standard grading and domain 
$$((\dom(D_1)\otimes_B Y)\cap \dom(\Diag(D_2)))\otimes \C^2 .$$
\item  If $(X,D_1)$ is odd and $(Y,D_2)$ is even,
\[D_1\otimes \gamma_2+1\otimes_\nabla D_2\]
on $X\hat\otimes_BY$ equipped with the trivial grading and domain $(\dom(D_1)\otimes_B Y)\cap \dom(\Diag(D_2))$.
\end{enumerate}
\end{Def}
It follows as in Theorem 5.5 in \locCit that this defines selfadjoint and regular operators if the connexion is Hermitian.  

For the definition of correspondence, we refer the reader to Definition 6.3 of \locCit However when $(X,D_1)$ is even and $(Y,D_2)$ is odd we replace in condition $(4)$ the operator $D_1\otimes 1$ by the operator $D_1\otimes\gamma$.

Using essentially the same techniques as in \locCit, one obtains (compare Theorem 6.7 and Theorem 7.5 therein):
\begin{Th}\label{Theorem:products} Suppose that $x=(X,D_1)$, $y=(Y,D_2)$ admits a correspondence $(X_1,\nabla^0)$, then for any (completely bounded) Hermitian $D_2$-connexion $\nabla$, the Fredholm module with operator $D_1\times_\nabla D_2$ defines a class $z$ in $KK$ which is an unbounded representative of the product of $x$ and $y$.
\end{Th}
\begin{Rem} That is, the bounded transform $q(\check z)$ is a Kasparov product of $q(\check x)$ and $q(\check y)$.
\end{Rem}
\begin{proof} 
It remains to prove the case of $x$ is odd and $y$ is even, the case where both modules are odd being shown in \locCit We first need to check that $(X\otimes_B Y,D_1\times_\nabla D_2)$  defines an unbounded Fredholm module. First of all, the operator $1\otimes_\nabla D_2$ (compare Theorem 5.4. of \locCit and \cite{NCG} Chapter 6, Section 3, Lemma 1) defined  by 
$$
(1\otimes_\nabla D_2)(\xi\otimes\eta)=\xi\otimes D_2\eta+\nabla_\xi(\eta)
$$
is selfadjoint and regular by Theorem 5.4 in \cite{SpecFlowKL}. Now as we have replaced condition $(4)$ in the definition of correspondence, we may apply Theorem 1.3 of \locCit  to conclude that $D_1\otimes\gamma+1\otimes_{\nabla^0}D_2$ is selfadjoint and regular. So this follows also for the operator $D_1\times_{\nabla}D_2$, as it is a bounded perturbation of the former operator. It follows as in the proof of Theorem 6.7 in \locCit that we actually obtain a Kasparov module.

Let us now show that this Fredholm module actually represents the Kasparov product of $(X,D_1)$ and $(Y,D_2)$. Since the analytic details are almost \emph{verbatim} the same as in \locCit, we will only treat the different algebraic aspects. As in the proof of Theorem 7.5. in \locCit, we may assume that $\nabla=\nabla_0$ by a perturbation argument. Denote by $\gamma$ the grading operator on $\mathbbm{C}l_1$, by $Z:=(\mathbbm{C}l_1\otimes X)\tilde\otimes_B Y$ the tensor product of $\mathbbm{C}l_1\otimes X$ with $Y$ over $B$ with grading operator $\gamma\otimes 1\otimes 1$, and by $Z'$ the same tensor product with grading operator $\gamma\otimes 1\otimes\gamma_2$ (note that we view $Z$ and $Z'$ as a $(\mathbbm{C}l_1\otimes A,C)$-correspondence). Let $c=(1,-1)\in\mathbbm{C}l_1$, $d:=-c$ and define a graded unitary isomorphism $\theta:Z'\to Z$ by setting
\[
\theta(x\otimes \xi\otimes\eta):=
\begin{cases}
x\otimes \xi\otimes\eta&\text{ if } \partial(x)=0=\partial(\eta)\\
cx\otimes\xi\otimes\eta&\text{ if }\partial(x)=1=\partial(\eta)\\
x\otimes \xi\otimes\eta&\text{ if }\partial (x)=1,\partial(\eta)=0\\
dx\otimes\xi\otimes\eta&\text{ if }\partial(x)=0,\partial(\eta)=1
\end{cases}\]
Let $D:=c\otimes D_1\times_{\nabla^0} D_2$, $D':=\theta^{-1}D\theta$ and $x:=(Z,D)$. We will show that $\theta^*\theta^{-1}_*(\check x)=(Z',D')$ is a representative of the Kasparov product $\check x\otimes_B y$. Note that $D'$ defines an odd operator on $Z'$, $\theta$ being a graded isomorphism and $c\otimes D_1\times_{\nabla^0} D_2$ an odd operator on $Z$. As in \locCit,  we have to check the conditions of \cite{Kucpub}, Theorem 13. The connexion condition (i) follows by a case-by-case calculation. For example, assume that $v\otimes \xi\in \mathbbm{C}l_1\otimes X$ with $\partial v=1$, $\xi\in\dom(D_1)\cap X_1$ and $\partial \eta=0$, then if we have for all $\xi\in\dom(D_2)$:
\begin{align*}
&T_{v\otimes\xi} D_2(\eta)-(-1)^{\partial v} D'(v\otimes\xi\otimes \eta)\\
=&T_{v\otimes\xi} D_2(\eta)+\theta^{-1}(c\otimes D_1\times_{\nabla^0} D_2)\theta(v\otimes\xi\otimes \eta)\\
=&v\otimes \xi\otimes D_2\eta+\theta^{-1}(c\otimes D_1\otimes\gamma+c\otimes 1\otimes_{\nabla^0}D_2)(v\otimes\xi\otimes \eta)\\
=&v\otimes\xi\otimes D_2\eta+\theta^{-1}(cv\otimes\xi\otimes D_2\eta+cv\otimes\nabla^0_\xi\eta)+\theta^{-1}(cv\otimes D_2\xi\otimes\eta)\\
=&v\otimes\xi\otimes D_2\eta+dcv\otimes\xi\otimes D_2\eta+\theta^{-1}(cv\otimes \nabla^0_\xi\eta+cv\otimes D_2\xi\otimes\eta)\\
=&\theta^{-1}(cv\otimes \nabla^0_\xi\eta+cv\otimes D_2\xi\otimes\eta).
\end{align*}
The other cases are similar. We leave the remaining (analytic) details to the reader -- these follow as in Section 7 of \locCit
\end{proof} 
\paragraph{Acknowledgements}

The authors are grateful to G. {Skandalis} for discussions and remarks, T. Masson for references regarding derivations and connexions, and V. Alekseev for interesting discussions. We would also like to thank J. Kaad for numerous discussions and advice concerning the article \cite{SpecFlowKL}.

\medbreak

Olivier \textsc{Gabriel}, \texttt{ogabriel@uni-math.gwdg.de}, 

Mathematisches Institut -- Universität Göttingen

Bunsenstr. 3-5 D--37 073 Göttingen, Germany

\bigbreak

Martin \textsc{Grensing}, \texttt{martin@grensing.net}

D\'epartement de Math\'ematiques -- Universit\'e d'Orl\'eans, 

B.P. 6759 -- 45 067 Orl\'eans cedex 2, France.

	\bibliographystyle{alpha} 
	\bibliography{biblio}

\end{document}